\documentclass[10pt,a4paper]{article}%
\usepackage{amsmath}
\usepackage{graphicx}
\usepackage{amsfonts}
\usepackage{amssymb}
\usepackage{hyperref}%
\providecommand{\U}[1]{\protect\rule{.1in}{.1in}}
\providecommand{\U}[1]{\protect\rule{.1in}{.1in}}
\newtheorem{theorem}{Theorem}

\newtheorem{definition}[theorem]{Definition}

\newtheorem{lemma}[theorem]{Lemma}

\newtheorem{proposition}[theorem]{Proposition}

\newenvironment{proof}[1][Proof]{\noindent\textbf{#1.} }{\ \rule{0.5em}{0.5em}}
\begin{document}

\author{Ognyan Kounchev
\and Hermann Render
\and Tsvetomir Tsachev}
\title{Error estimates for harmonic and biharmonic interpolation splines with annular geometry}
\date{}
\maketitle

\begin{abstract}
The main result in this paper is an error estimate for interpolation
biharmonic polysplines in an annulus $A\left(  r_{1},r_{N}\right)  $, with
respect to a partition by concentric annular domains $A\left(  r_{1}%
,r_{2}\right)  ,$ ...., $A\left(  r_{N-1},r_{N}\right)  ,$ for radii
$0<r_{1}<....<r_{N}.$ The biharmonic polysplines interpolate a smooth function
on the spheres $\left\vert x\right\vert =r_{j}$ for $j=1,...,N$ and satisfy
natural boundary conditions for $\left\vert x\right\vert =r_{1}$ and
$\left\vert x\right\vert =r_{N}.$ By analogy with a technique in
one-dimensional spline theory established by C. de Boor, we base our proof on
error estimates for harmonic interpolation splines with respect to the
partition by the annuli $A\left(  r_{j-1},r_{j}\right)  $. For these estimates
it is important to determine the smallest constant $c_{d}\left(
\Omega\right)  ,$ where $\Omega=A\left(  r_{j-1},r_{j}\right)  ,$ among all
constants $c$ satisfying
\[
\sup_{x\in\Omega}\left\vert f\left(  x\right)  \right\vert \leq c\sup
_{x\in\Omega}\left\vert \Delta f\left(  x\right)  \right\vert
\]
for all $f\in C^{2}\left(  \Omega\right)  \cap C\left(  \overline{\Omega
}\right)  $ vanishing on the boundary of the bounded domain $\Omega$ . In this
paper we describe $c_{d}\left(  \Omega\right)  $ for an annulus $\Omega
=A\left(  r,R\right)  $ and we will give the estimate
\[
\min\{\frac{1}{2d},\frac{1}{8}\}\left(  R-r\right)  ^{2}\leq c_{d}\left(
A\left(  r,R\right)  \right)  \leq\max\{\frac{1}{2d},\frac{1}{8}\}\left(
R-r\right)  ^{2}%
\]
where $d$ is the dimension of the underlying space.

\end{abstract}


\section{Introduction}

Let $C\left(  \Omega\right)  $ be the set of all continuous complex-valued
functions defined on a subset $\Omega$ in $\mathbb{R}^{d}.$ For an open subset
$\Omega,$ let $C^{m}\left(  \Omega\right)  $ be the set of all functions which
have continuous partial derivatives of order $\leq m$ on $\Omega$, and for
$\overline{\Omega}$, the closure of $\Omega$, we denote by $C^{m}\left(
\overline{\Omega}\right)  $ the set of all functions which have continuous
partial derivatives of order $\leq m$ on $\overline{\Omega}.$

Recall that a \emph{spline of degree} $p$ defined on an interval $\left[
t_{1},t_{N}\right]  $ with nodes $t_{1}<...<t_{N}$ is a function $S\in
C^{p-1}\left[  t_{1},t_{N}\right]  $ which on each subinterval $\left(
t_{j},t_{j+1}\right)  $ is identical with some polynomial of degree $\leq p$
for $j=1,...,N-1.$ If $p=3$ we call $S$ a \emph{cubic spline}. For $p=1$ we
obtain the definition of a \emph{linear spline}. Note that $S$ is a linear
spline if and only if $S$ a continuous function on $\left[  t_{1}%
,t_{N}\right]  $ such that on each subinterval $\left(  t_{j},t_{j+1}\right)
$ the function is linear, so a solution of the differential operator
\begin{equation}
\frac{d^{2}}{dt^{2}}S\left(  t\right)  =0. \label{eqsecond}%
\end{equation}

The aim of the present paper is to provide error estimates for interpolation
by special types of multivariate splines, namely \emph{harmonic} and
\emph{biharmonic} splines. Harmonic splines occur in a natural fashion in
mathematical problems, see e.g. the discussion in \cite{Hayman}. Harmonic
splines \emph{for block partitions} have been discussed by various authors in
the literature, see \cite{BaLe13}, \cite{BaLe14}, \cite{BLM10}, \cite{Klim95}.
Here a set $\Omega$ of the form $\left[  a_{1},b_{1}\right]  \times
\cdots\times\left[  a_{d},b_{d}\right]  $ is called a block in $\mathbb{R}%
^{d}$, and a harmonic spline is continuous function on $\Omega$ which is
harmonic on open and disjoint subdomains (called subblocks) $\Omega_{j}$ of
$\Omega$ for $j=1,...,N$ such that the closure of $\Omega_{1}\cup\ldots
\cup\Omega_{N}$ is equal to $\Omega.$ Recall that a function $f:\Omega
\rightarrow\mathbb{C}$ is called \emph{harmonic }if $f\in C^{2}\left(
\Omega\right)  $ and $\ $%
\begin{equation}
\Delta f\left(  x\right)  :=\sum_{j=1}^{d}\frac{\partial^{2}}{\partial
x_{j}^{2}}f\left(  x\right)  =0\text{ for all }x\in\Omega, \label{eqdelta}%
\end{equation}
and $\Delta$ is the Laplace operator. It is convenient to introduce the
following natural generalization: we say that a function $S$ is a
\emph{harmonic spline with respect to the partition} $\Omega_{1}%
,....,\Omega_{N}$ if $S\in C\left(  \overline{\Omega}\right)  $ for
\[
\Omega=\Omega_{1}\cup\ldots\cup\Omega_{N}%
\]
and $S$ is a harmonic function on each subdomain $\Omega_{j}$ for $j=1,...N.$
Thus a harmonic spline for a partition $\Omega_{1},...,\Omega_{N}$ is a
multivariate generalization of a linear spline.

For the partition\emph{ }$\Omega_{1},...,\Omega_{N},$ we define harmonic
spline interpolation in the following way: Given a function $F\in C\left(
\overline{\Omega}\right)  $, we say $I_{2}\left(  F\right)  :\overline{\Omega
}\rightarrow\mathbb{R}$ is a \emph{harmonic spline interpolating the data
function }$F$\emph{ }if $I_{2}\left(  F\right)  $ is a harmonic spline for
$\Omega_{1},...,\Omega_{N}$ and it satisfies the interpolation condition
\begin{equation}
I_{2}\left(  F\right)  \left(  \xi\right)  =F\left(  \xi\right)  \text{ for
all }\xi\in\partial\Omega=\partial\Omega_{1}\cup...\cup\partial\Omega_{N}.
\label{eqintertransfinite}%
\end{equation}
Note that we require in (\ref{eqintertransfinite}) an interpolation condition
for infinitely many points, and in the literature this is often called
\emph{transfinite interpolation}. Transfinite interpolation was already
considered by Gordon and Hall in \cite{GoHa73} (cf. \cite{DJM20}) and this
concept has found many applications in mesh generation, geometric modelling,
finite element methods and spline analysis, see \cite{Bejan11}, \cite{Bejan13}%
, \cite{DyFl09}, \cite{Sabi94}, \cite{Skor18}.

The existence of a harmonic spline interpolant is easy to establish when we
assume that the Dirichlet problem is solvable\footnote{The Dirichlet problem
is solvable for a domain $\Omega$ if for each continuous function $f$ defined
on the boundary $\partial\Omega$ of $\Omega$ there exists a continuous
function $h$ defined on the closure $\overline{\Omega}$ of $\Omega$ which is
harmonic in $\Omega$ and interpolates $f$ on the boundary, i.e. $h\left(
\xi\right)  =f\left(  \xi\right)  $ for all $\xi\in\partial\Omega.$} for each
domain $\Omega_{j}$ for $j=1,...,N.$ The uniqueness of the harmonic
interpolant is a simple consequence of the maximum principle. The main results
in \cite{BLM10}, \cite{BaLe13}, \cite{BaLe14}, \cite{Klim95} are optimal
estimates for the error
\[
E\left(  F\right)  :=F-I_{2}\left(  F\right)
\]
for a given twice differentiable function $F\in C^{2}\left(  \overline{\Omega
}\right)  $ in the case of\emph{ block partitions} with respect to the
supremum norm defined by
\[
\left\Vert f\right\Vert _{\Omega}=\sup_{x\in\Omega}\left\vert f\left(
x\right)  \right\vert
\]
for $f\in C\left(  \Omega\right)  $. In the present paper we shall present an
explicit error estimate for harmonic splines for a partition given by annular
subdomains of the form \emph{ }
\[
\Omega_{j}:=A\left(  r_{j},r_{j+1}\right)  :=\left\{  x\in\mathbb{R}^{d}%
:r_{j}<\left\vert x\right\vert <r_{j+1}\right\}  \text{ }%
\]
for given positive radii $r_{1}<...<r_{N}$. Indeed, we shall proof the following:

\begin{theorem}
\label{ThmA1}Let $r_{1}<....<r_{N}$ be positive numbers, and let $F\in
C^{2}(\overline{A\left(  r_{1},r_{N}\right)  }).$ Assume that $I_{2}\left(
F\right)  $ is the harmonic spline interpolating $F\ $ for all $x$ with
$\left\vert x\right\vert =r_{j}$ and $j=1,...n.$ Then
\begin{equation}
\left\Vert F-I_{2}\left(  F\right)  \right\Vert _{A\left(  r_{1},r_{N}\right)
}\ \leq C_{d}\max_{j=1,...,N-1}\text{ }\left(  r_{j+1}-r_{j}\right)
^{2}\left\Vert \Delta F\right\Vert _{A\left(  r_{1},r_{N}\right)  }\ .
\label{eqHS1}%
\end{equation}
where $d$ is the dimension of the space and $C_{d}=\max\{\frac{1}{2d},\frac
{1}{8}\}.$
\end{theorem}

Note that this result is very similar to the error estimate for linear splines
(see e.g. \cite[p. 31]{deBoor}): Assume that $S_{1}\left(  F\right)  $ is the
(unique) linear spline interpolating a twice differentiable function
$F:\left[  t_{1},t_{N}\right]  \rightarrow\mathbb{R}$ at the points
$t_{1},...,t_{N}.$ Then
\begin{equation}
\sup_{t\in\left[  t_{1},t_{N}\right]  }\left\vert F\left(  t\right)
-S_{1}\left(  F\right)  \left(  t\right)  \right\vert \ \ \ \ \leq\frac{1}%
{8}\max_{j=1,..N-1}\left\vert t_{j+1}-t_{j}\right\vert ^{2}\cdot\sup
_{t\in\left[  t_{1},t_{N}\right]  }\left\vert F^{\prime\prime}\left(
t\right)  \right\vert . \label{eqClass1}%
\end{equation}
The most difficult part in Theorem \ref{ThmA1} is to establish the explicit
nature of the constant $C_{d}$ in the estimate (\ref{eqHS1}). Let us also
emphasize that Theorem \ref{ThmA1} is an important ingredient to establish an
error $L^{2}$-estimate for interpolation with biharmonic splines -- the next
topic we want to discuss.

The definition of a harmonic spline can be traced back in old sources (see.g.
\cite{Hayman}), and the concept is very intuitive. For the definition of a
biharmonic spline we use the approach given in \cite{kounchevbook} where the
explicit description as a piecewise polyharmonic function is used and which
emphasizes the analogy to the case of univariate cubic splines. For our
purposes it is convenient to use the following definition:

\begin{definition}
Let $\Omega_{1},...,\Omega_{N}$ be open disjoint sets in $\mathbb{R}^{d}$ and
define $\Omega=\Omega_{1}\cup\ldots\cup\Omega_{N}.$ A function $f:\overline
{\Omega}\rightarrow\mathbb{C}$ is a \emph{biharmonic spline for the partition}
$\Omega_{1},...,\Omega_{N}$ if $f\in C^{2}\left(  \overline{\Omega}\right)  $
and the restriction of $f$ to each $\Omega_{j}$ is biharmonic, i.e. $f\in
C^{4}\left(  \Omega_{j}\right)  $ and
\begin{equation}
\Delta^{2}f\left(  x\right)  :=\Delta\circ\Delta f\left(  x\right)  =0
\label{eqbiharmonic}%
\end{equation}
for all $x\in\Omega_{j}$ and for $j=1,....,N$.
\end{definition}

In the definition of a biharmonic spline the matching of the boundary
behaviour of the biharmonic functions defined on $\Omega_{j}$ is simply
expressed by the requirement that $f$ is a $C^{2}$-function on $\overline
{\Omega}$. This corresponds to the definition of a cubic spline on the
interval $\left[  t_{1},t_{N}\right]  $ with nodes $t_{1}<...<t_{N}$: it is a
function $S\in C^{2}\left[  t_{1},t_{N}\right]  $ which on each subinterval
$\left(  t_{j},t_{j+1}\right)  $ is identical with a solution of the
differential equation $\frac{d^{4}}{dt^{4}}S=0$ for $j=1,...,N-1.$

The existence of an interpolating biharmonic spline requires additional
assumptions on the smoothness of the domains $\Omega_{j}$ and higher
regularity of the data function on the boundary $\partial\Omega$ of $\Omega$
which might be expressed in terms of H\"{o}lder or Sobolev spaces. On the
other hand, we are interested only in error estimates, so at this place we do
not need to dwell in the more difficult question of the existence of an
interpolation biharmonic spline for a partition $\Omega_{1},...,\Omega_{n}$
(see also Section 4 for some comments in the case of annuli, and for the
details on the existence of interpolation polysplines consult the monograph
\cite{kounchevbook}, chapter $22$).

Our main result is the following error estimate where we recall that the
$L^{p}$-norm of a measurable function $f:\Omega\rightarrow\mathbb{C}$ is
defined by
\[
\text{ }\left\Vert f\right\Vert _{L^{p}\left(  \Omega\right)  }=\left(
\int_{\Omega}\left\vert f\left(  x\right)  \right\vert ^{p}dx\right)
^{\frac{1}{p}}.
\]

\begin{theorem}
\label{Thm1}Let $0<r_{1}<...<r_{N}$ and let $F\in C^{4}(\overline{A\left(
r_{1},r_{N}\right)  })$. Assume that $I_{4}\left(  F\right)  $ is a biharmonic
spline for the partition $A\left(  r_{1},r_{2}\right)  ,....,A\left(
r_{N},r_{N-1}\right)  $ which satisfies the (transfinite) interpolation
conditions
\begin{equation}
I_{4}\left(  x\right)  =F\left(  x\right)  \text{ for all }x\text{ with
}\left\vert x\right\vert =r_{j}\text{ and }j=1,...,N, \label{eqFinterpolating}%
\end{equation}
and the boundary conditions for the normal derivative $\frac{\partial
}{\partial n}$
\begin{equation}
\frac{\partial I_{4}\left(  F\right)  }{\partial n}\left(  x\right)
=\frac{\partial F\left(  x\right)  }{\partial n}\text{ for all }\left\vert
x\right\vert =r_{1}\text{ and for all }\left\vert x\right\vert =r_{N}.
\label{eqFcomplete}%
\end{equation}
Then, with $C_{d}=\max\left\{  \frac{1}{2d},\frac{1}{8}\right\}  $ as above,
the following estimate holds:
\[
\left\Vert F-I_{4}\left(  F\right)  \right\Vert _{L^{2}\left(  A\left(
r_{1},r_{N}\right)  \right)  }\leq C_{d}^{2}\cdot\max_{j=1,...,N-1}\left\vert
r_{j+1}-r_{j}\right\vert ^{4}\left\Vert \Delta^{2}F\right\Vert _{L^{2}\left(
A\left(  r_{1},r_{N}\right)  \right)  }.
\]

\end{theorem}

The above result should be compared with the $L^{2}$-error estimate of a
one-dimensional cubic spline: Assume that $F\in C^{4}\left[  t_{1}%
,t_{N}\right]  $ and let $S_{3}\left(  F\right)  $ be a cubic spline
interpolating $F$ at the points $t_{1},...,t_{N}$ and satisfying the
additional boundary condition
\[
\frac{d}{dt}F\left(  t_{1}\right)  =\frac{d}{dt}S_{3}\left(  F\right)  \left(
t_{1}\right)  \text{ and }\frac{d}{dt}F\left(  t_{N}\right)  =\frac{d}%
{dt}S_{3}\left(  F\right)  \left(  t_{N}\right)  .
\]
Then \
\[
\left\Vert F-S_{3}\left(  F\right)  \right\Vert _{L^{2}\left(  t_{1}%
,t_{N}\right)  }\leq\frac{4}{\pi^{4}\ }\max_{j=1,...,N}\left\vert
t_{j+1}-t_{j}\right\vert ^{4}\left\Vert \frac{d^{4}F}{dt^{4}}\right\Vert
_{L^{2}\left(  t_{1},t_{N}\right)  },
\]
see e.g. \cite{ADL09}. Let us also mention that for the supremum norm the
following error estimate
\begin{equation}
\max_{t\in\left[  t_{1},t_{N}\right]  }\left\vert F\left(  t\right)
-S_{3}\left(  F\right)  \left(  t\right)  \right\vert \leq\frac{1}{16}%
\max_{j=1,...,N}\left\vert t_{j+1}-t_{j}\right\vert ^{4}\max_{x\in\left[
t_{1},t_{N}\right]  }\left\vert \frac{d^{4}F}{dt^{4}}\left(  x\right)
\right\vert \label{ineqmax}%
\end{equation}
holds, see \cite[p. 55]{deBoor}. We leave the question open whether in Theorem
\ref{Thm1} one may replace the $L^{2}$-norm by the supremum norm. In passing
we mention that in \cite{KoRe20} the inequality (\ref{ineqmax}) has been
generalized to $L$-splines where $L$ is a differential operator with constant
coefficients of order $4.$

Let us briefly describe the structure of the paper: in Section 2 we shall
discuss error estimate for harmonic interpolation splines with respect to a
general partition $\Omega_{1},...,\Omega_{N}$. This problem is closely related
to the problem of finding the smallest constant $c_{d}\left(  \Omega\right)  $
among all constants $c$ which satisfy the inequality
\begin{equation}
\left\Vert f\right\Vert _{\Omega}\leq c\left\Vert \Delta f\right\Vert
_{\Omega} \label{eqSmallesta}%
\end{equation}
for all $f\in C\left(  \overline{\Omega}\right)  \cap C\left(  \Omega\right)
$ \emph{vanishing on the boundary} $\partial\Omega.$ In Section 2 we will
characterize the constant $c_{d}\left(  \Omega\right)  $:

\begin{theorem}
\label{Remark1}Let $\Omega$ be a bounded regular domain and let $h$ be the
solution of the Dirichlet problem for the data function $\left\vert
x\right\vert ^{2}.$ Then
\begin{equation}
c_{d}\left(  \Omega\right)  =\sup_{x\in\Omega}T_{0}\left(  x\right)  \text{
and }T_{0}\left(  x\right)  =\frac{1}{2d}\left(  h\left(  x\right)
-\left\vert x\right\vert ^{2}\right)  . \label{eqDefFoo}%
\end{equation}
The function $T_{0}$ is the unique function which vanishes on the boundary of
$\Omega$ and satisfies $\Delta T_{0}=-1.$
\end{theorem}

The function $T_{0}$ plays an eminent role in various area of mathematics and
is called the torsion function, see e.g. the fundamental work of G. P\'{o}lya,
G. Szeg\"{o} about isoperimetric inequalities in \cite{PoSz51}, or the
monography \cite{Sperb81}. There is a vast literature on this subject with
many ramifications and it would take too much space to survey the results, so
we only mention a very incomplete list of new references \cite{BCH98},
\cite{Cabre95}, \cite{Cabre17}, \cite{CaRa11}, \cite{CaRuTa10}. In Section 2
we shall provide a self-contained proof of Theorem \ref{Remark1} which is
based on a Green function approach.

In Section 3 we provide the proof of Theorem \ref{Thm1}. It is remarkable that
the $L^{2}$ -error estimate in the biharmonic case can be performed by an
iterative argument where the $L^{2}$-error estimate for harmonic splines is
used twice.

In Section 4 we present a proof of an orthogonality relation for biharmonic
interpolation splines which is used in Section 3.

In Section 5 we provide a computation of the best constant $c_{d}\left(
\Omega\right)  $ for the \emph{annular domain } $A\left(  r,R\right)
=\left\{  x\in\mathbb{R}^{d}:r<\left\vert x\right\vert <R\right\}  $ and we
prove the following inequalities:
\[
\min\{\frac{1}{2d},\frac{1}{8}\}\left(  R-r\right)  ^{2}\leq c_{d}\left(
A\left(  r,R\right)  \right)  \leq\max\{\frac{1}{2d},\frac{1}{8}\}\left(
R-r\right)  ^{2}.
\]

Finally, let us mention that some of the presented concepts can be
generalized. Recall that a function $f:\Omega\rightarrow\mathbb{C}$ is called
\emph{polyharmonic of order} $p$ if $f\in C^{2p}\left(  \Omega\right)  $ and
\[
\Delta^{p}f\left(  x\right)  =0\text{ for all }x\in\Omega
\]
where $\Delta^{p}$ is the $p$-th iterate of $\Delta,$ see \cite{ACL83},
\cite{Avan85}, \cite{GGS10}. Polyharmonic functions are often used in applied
mathematics, see e.g. \cite{BBRV05}, \cite{DKLR11},
\cite{kounchevRenderKleinDirac}, \cite{KoRe13},  \cite{KoReMathNach},
\cite{MaNe90}, \cite{Rabut1992}, \cite{Rabut2008}. Slightly more general than
in \cite{kounchevbook} we define a function $f:\overline{\Omega}%
\rightarrow\mathbb{C}$ to be \emph{polyspline of order }$p$ \emph{for a
partition }$\Omega_{1},...,\Omega_{N}$ if $f\in C^{2p-2}\left(  \overline
{\Omega}\right)  $ for $\Omega=\Omega_{1}\cup\ldots\cup\Omega_{N}$ and
$\Delta^{p}f\left(  x\right)  =0$ for all $x\in\Omega_{j}$ and for
$j=1,....,N$. Cardinal polysplines of order $p$ on strips or annuli have been
discussed by the first two authors in a series of papers
\cite{kounchevrenderPAMS},  \cite{KoReCM}, \cite{kounchevrenderJAT}. 

In this paper we have dealt with transfinite interpolation and it might be of
interest to compare our results with the thin plate splines of order $p>d/2$
(in $\mathbb{R}^{d})$ introduced by J. Duchon in \cite{Duchon77} for the
interpolation at a finite number of scattered points $x_{1},...,x_{N}%
\in\mathbb{R}^{d}$. Thin plate splines are polyharmonic functions of order $p$
on the set $\mathbb{R}^{d}\diagdown\left\{  x_{1},....,x_{N}\right\}  $ since
they are a finite linear combination of translates of the fundamental solution
of $\Delta^{p}$ in $\mathbb{R}^{d}.$ In contrast to a polyspline a thin plate
spline is only a function in $C^{2p-d-1}\left(  \mathbb{R}^{d}\right)  $ which
is the reason for the requirement $p>d/2.$ By definition, an interpolating
thin plate spline is defined as the unique minimizer of the integral
functional%
\[
\left\langle f,f\right\rangle _{p,2,\mathbb{R}^{d}}:=\int_{\mathbb{R}^{d}}%
{\displaystyle\sum_{\left\vert \alpha\right\vert =p}}
\frac{p!}{\alpha!}D^{\alpha}f\left(  x\right)  \cdot\overline{D^{\alpha
}f\left(  x\right)  }dx
\]
among all functions $f:\mathbb{R}^{d}\rightarrow\mathbb{R}$ having all partial
derivative $D^{\alpha}f$ of total order $\left\vert \alpha\right\vert =p$ in
$L^{2}\left(  \mathbb{R}^{d}\right)  $ and interpolating the data. Here we
used multi-index notation $\alpha=\left(  \alpha_{1},...,\alpha_{d}\right)
\in\mathbb{N}_{0}^{d}$ with $\left\vert \alpha\right\vert =\alpha_{1}%
+\cdots+\alpha_{d}$ and $\alpha!=\alpha_{1}!\cdots\alpha_{d}!$ and $D^{\alpha
}=\frac{\partial^{\alpha_{1}}}{\partial x_{1}^{\alpha_{1}}}...\frac
{\partial^{\alpha_{d}}}{\partial x_{d}^{\alpha_{d}}}$. In \cite{Duchon78} one
can find the error estimates which served as model example for error estimates
for interpolation with radial basis functions, see e.g.  \cite{buhmann},
\cite{NaWa04} \cite{Wendland}, \cite{WuSchaback}. Unlike our results, in these
references the constants for the error estimates are not explicit, and only
Sobolev norms are used; only in two dimensions some explicit constants are
found, cf. \cite{Powell1990}, \cite{Powell1994}. It is our expectation that
the error estimates for polysplines can be used to improve the error estimates
for thin plates for data which is structured along curves, a subject we hope
to address in a future paper.

\section{Harmonic interpolation splines}

We say that an open set $\Omega$ is \emph{regular} in $\mathbb{R}^{d}$ if each
boundary point is regular, see \cite[p. 179]{ArGa} for definition, and
\cite[Theorem 6.5.5]{ArGa} for a characterization. It is known that for a
regular bounded domain $\Omega$ the Dirichlet problem is solvable, see
\cite[Theorem 6.5.5 and 6.5.4]{ArGa}.

At first we discuss the error estimate for harmonic interpolation splines:

\begin{theorem}
\label{ThmMAXGeneral}Assume that $\Omega_{1},....\Omega_{N}$ are pairwise
disjoint bounded regular domains and define $\Omega=\cup_{j=1}^{N}\Omega_{j}.$
If $F\in C^{2}\left(  \overline{\Omega}\right)  $ and $I_{2}\left(  F\right)
$ is the harmonic spline interpolating $F$ on $\partial\Omega$ then
\[
\left\Vert F-I_{2}\left(  F\right)  \right\Vert _{\Omega}\leq\max
_{j=1,...,N}\text{ }c_{d}\left(  \Omega_{j}\right)  \ \cdot\sup_{y\in\Omega
}\left\vert \Delta F\left(  y\right)  \right\vert
\]
where $c_{d}\left(  \Omega_{j}\right)  $ is the smallest constant $c$ such
that (\ref{eqSmallesta}) holds for all function $f\in C\left(  \overline
{\Omega_{j}}\right)  \cap C^{2}\left(  \Omega_{j}\right)  $ vanishing on the
boundary $\partial\Omega_{j}$ for $j=1,...,N$.
\end{theorem}

\begin{proof}
We consider $f\left(  x\right)  =F\left(  x\right)  -I_{2}\left(  F\right)
\left(  x\right)  .$ Then $f\left(  x\right)  =0$ for all $x\in\partial
\Omega_{j}\subset\partial\Omega.$ Further $f\in C\left(  \overline{\Omega_{j}%
}\right)  \cap C^{2}\left(  \Omega_{j}\right)  $ since $I_{2}\left(  F\right)
$ is harmonic on $\Omega_{j}$ and continuous on $\overline{\Omega_{j}}.$ Hence
for $x\in\Omega_{j}$%
\begin{equation}
\left\vert F\left(  x\right)  -I_{2}\left(  F\right)  \left(  x\right)
\right\vert \leq\ c_{d}\left(  \Omega_{j}\right)  \cdot\sup_{y\in\Omega_{j}%
}\left\vert \Delta F\left(  y\right)  \right\vert \text{ } \label{eqlast1}%
\end{equation}
The statement is now obvious.
\end{proof}

A fundamental theorem in potential theory states that for any open set $U$ in
$\mathbb{R}^{d}$ with $d\geq3,$ or for any bounded open set $U$ in
$\mathbb{R}^{2},$ the Green function $G_{\Omega}\left(  x,y\right)  $ exists,
see \cite[p. 90]{ArGa}. Further we denote by $\omega_{d}$ the volume of the
unit ball and define
\[
a_{d}=\left\{
\begin{array}
[c]{cc}%
\frac{1}{2\pi} & \text{for }d=2\\
\frac{1}{d\left(  d-2\right)  \omega_{d}} & \text{for }d\geq3.
\end{array}
\right.
\]

\begin{theorem}
If $\Omega$ is a bounded domain in $\mathbb{R}^{d}$ then
\begin{equation}
c_{d}\left(  \Omega\right)  \leq a_{d}\sup_{x\in\Omega}\int_{\Omega}G_{\Omega
}\left(  x,y\right)  dy. \label{eqDefCOmega}%
\end{equation}
Equality holds when $\Omega$ is a regular domain.
\end{theorem}

\begin{proof}
Assume that $f\in C\left(  \overline{\Omega}\right)  \cap C\left(
\Omega\right)  $ vanishes on the boundary $\partial\Omega.$ If $\Delta f$ is
unbounded on $\Omega$ the inequality (\ref{eqSmallesta}) is trivial. If
$\Delta f$ is bounded it follows that
\begin{equation}
\int_{\Omega}\left\vert \Delta f\left(  y\right)  \right\vert dy<\infty\text{
and }\int_{\Omega}\left\vert \Delta f\left(  y\right)  \right\vert
^{p}dy<\infty. \label{eqIntegrable}%
\end{equation}
for some $p>d/2.$ Then it can be shown that the following representation
formula holds
\begin{equation}
f\left(  x\right)  =a_{d}\int_{\Omega}G_{\Omega}\left(  x,y\right)  \Delta
f\left(  y\right)  dy \label{eqRepHarmonic}%
\end{equation}
for all $x\in\Omega$, and all $f\in C^{2}\left(  \Omega\right)  \cap C\left(
\overline{\Omega}\right)  $ which vanish on the boundary $\partial\Omega.$
Clearly (\ref{eqRepHarmonic}) implies that
\begin{equation}
\left\vert f\left(  x\right)  \right\vert \leq\sup_{y\in\Omega}\left\vert
\Delta f\left(  y\right)  \right\vert \text{ }\cdot a_{d}\int_{\Omega
}G_{\Omega}\left(  x,y\right)  dy\text{ } \label{eq1}%
\end{equation}
for any $x\in\Omega.$ The first result follows since $c_{d}\left(
\Omega\right)  $ is the smallest number satisfying (\ref{eqSmallesta}).

Now assume that $\Omega$ is regular, so the Dirichlet problem is solvable for
$\Omega$. Then there exists a harmonic function $h$ in $C\left(
\overline{\Omega}\right)  $ such that $h\left(  \xi\right)  =\left\vert
\xi\right\vert ^{2}$ for all $\xi\in\partial\Omega.$ Clearly the function
\begin{equation}
T_{0}\left(  x\right)  =\frac{1}{2d}\left(  h\left(  x\right)  -\left\vert
x\right\vert ^{2}\right)  \label{eqDefFnull}%
\end{equation}
is in $C\left(  \overline{\Omega}\right)  \cap C^{2}\left(  \Omega\right)  $
and vanishes on $\partial\Omega.$ Further $\Delta T_{0}\left(  x\right)  =-1.$
Since $c_{d}\left(  \Omega\right)  $ is the smallest constant satisfying
(\ref{eqSmallesta}) we infer that
\begin{equation}
\left\vert T_{0}\left(  x\right)  \right\vert \leq c_{d}\left(  \Omega\right)
\sup_{\tau\in\Omega}\left\vert \Delta T_{0}\left(  \tau\right)  \right\vert
=c_{d}\left(  \Omega\right)  . \label{eqF00}%
\end{equation}
The representation formula (\ref{eqRepHarmonic}) shows that
\begin{equation}
T_{0}\left(  x\right)  =-a_{d}\int_{\Omega}G_{\Omega}\left(  x,y\right)
\Delta T_{0}\left(  x\right)  dx=a_{d}\int_{\Omega}G_{\Omega}\left(
x,y\right)  dx. \label{eqFintegral}%
\end{equation}
It follows that from (\ref{eqDefCOmega}), formulae (\ref{eqFintegral}) and
(\ref{eqF00}) that
\[
c_{d}\left(  \Omega\right)  \leq a_{d}\sup_{x\in\Omega}\int_{\Omega}G_{\Omega
}\left(  x,y\right)  dx=\sup_{x\in\Omega}T_{0}\left(  x\right)  \leq
c_{d}\left(  \Omega\right)  .
\]
The proof is complete.
\end{proof}

Next we turn to $L^{p}$-estimates of harmonic splines:

\begin{theorem}
\label{ThmMainSecondS2}Let $\Omega$ be a bounded domain in $\mathbb{R}^{d}$.
Assume that $f\in C\left(  \overline{\Omega}\right)  \cap C^{2}\left(
\Omega\right)  $ vanishes on the boundary $\partial\Omega$ and that $\Delta f$
is bounded. Then for any $p>1$ and its conjugate exponent $q$ defined by
$\frac{1}{q}+\frac{1}{p}=1$ we have
\begin{equation}
\left\Vert f\right\Vert _{L^{p}\left(  \Omega\right)  }\leq\left(  \max
_{x\in\Omega}\int_{\Omega}a_{d}G_{\Omega}\left(  x,y\right)  dy\right)
^{\frac{p+q}{pq}}\cdot\left\Vert \Delta f\right\Vert _{L^{p}\left(
\Omega\right)  }. \label{eqineqL2}%
\end{equation}

\end{theorem}

\begin{proof}
The representation formula (\ref{eqRepHarmonic}) and the H\"{o}lder inequality
show that
\begin{align*}
\left\vert f\left(  x\right)  \right\vert  &  \leq\ \int_{\Omega}\left(
a_{d}G_{\Omega}\left(  x,y\right)  \right)  ^{\frac{1}{q}}\cdot\left(
a_{d}G_{\Omega}\left(  x,y\right)  \right)  ^{\frac{1}{p}}\left\vert \Delta
f\left(  y\right)  \right\vert dy\\
&  \leq\left(  \int_{\Omega}a_{d}G_{\Omega}\left(  x,y\right)  dy\right)
^{\frac{1}{q}}\left(  \int_{\Omega}a_{d}G_{\Omega}\left(  x,y\right)
\left\vert \Delta f\left(  y\right)  \right\vert ^{p}dy\right)  ^{\frac{1}{p}%
}.
\end{align*}
Write $S\left(  x\right)  =a_{d}\int_{\Omega}G_{\Omega}\left(  x,y\right)
dy,$ take the $p$-th power on both sides and integrate with respect to $x,$
then
\[
\int_{\Omega}\left\vert f\left(  x\right)  \right\vert ^{p}dx\leq\int_{\Omega
}\left(  S\left(  x\right)  \right)  ^{\frac{p}{q}}\int_{\Omega}a_{d}%
G_{\Omega}\left(  x,y\right)  \left\vert \Delta f\left(  y\right)  \right\vert
^{p}dydx.
\]
Then Fubini's theorem shows that
\begin{equation}
\int_{\Omega}\left\vert f\left(  x\right)  \right\vert ^{p}dx\leq
\ \int_{\Omega}\left(  \int_{\Omega}a_{d}G_{\Omega}\left(  x,y\right)  \left(
S\left(  x\right)  \right)  ^{\frac{p}{q}}dx\right)  \left\vert \Delta
f\left(  y\right)  \right\vert ^{p}dy. \label{eqIneq21}%
\end{equation}
Further we see that
\begin{align}
\int_{\Omega}a_{d}G_{\Omega}\left(  x,y\right)  \left(  S\left(  x\right)
\right)  ^{\frac{p}{q}}dx  &  \leq\max_{x\in\Omega}\left(  S\left(  x\right)
\right)  ^{\frac{p}{q}}\int_{\Omega}a_{d}G_{\Omega}\left(  x,y\right)
dx\label{eqDefF1}\\
&  \leq\max_{x\in\Omega}\left(  S\left(  x\right)  \right)  ^{\frac{p}{q}%
+1}=\max_{x\in\Omega}\left(  S\left(  x\right)  \right)  ^{\frac{p+q}{q}}%
\end{align}
Now take the $p$-th square root in (\ref{eqIneq21}) and we arrive at
\[
\left\Vert f\right\Vert _{L^{p}\left(  \Omega\right)  }\leq\max_{x\in\Omega
}\left(  S\left(  x\right)  \right)  ^{\frac{p+q}{pq}}\left\Vert \Delta
f\right\Vert _{L^{p}\left(  \Omega\right)  }.
\]

\end{proof}

We apply now the results to the case of annular domains.

\begin{theorem}
\label{ThmMainHarmSpAnn}Let $r_{1}<....<r_{N}$ be real numbers, and let $F\in
C^{2}(\overline{A\left(  r_{1},r_{N}\right)  }).$ Assume that $I_{2}\left(
F\right)  $ is a harmonic spline interpolating $F\ $ for all $x$ with
$\left\vert x\right\vert =r_{j}$ and $j=1,...N,$ then for the supremum norm
the estimate
\[
\left\Vert F-I_{2}\left(  F\right)  \right\Vert _{A\left(  r_{1},r_{N}\right)
}\ \leq C_{d}\max_{j=1,...,N-1}\text{ }\left(  r_{j+1}-r_{j}\right)
^{2}\left\Vert \Delta F\right\Vert _{A\left(  r_{1},r_{N}\right)  }\ ,
\]
holds, and for the $L^{2}$-norm
\[
\left\Vert F-I_{2}\left(  F\right)  \right\Vert _{L^{2}\left(  A\left(
r_{1},r_{N}\right)  \right)  }\leq C_{d}\max_{j=1,...,N-1}\left(
r_{j+1}-r_{j}\right)  ^{2}\left\Vert \Delta F\right\Vert _{L^{2}\left(
A\left(  r_{1},r_{N}\right)  \right)  }.
\]
where $C_{d}=\max\left\{  \frac{1}{2d},\frac{1}{8}\right\}  .$
\end{theorem}

\begin{proof}
The best constant $c_{d}\left(  \Omega\right)  $ for annular domains is
characterized and estimated in the last Section, see Theorem \ref{ThmAnnulus}.
Now Theorem \ref{ThmMAXGeneral} yields the first statement. For the second
statement we note that
\[
\left\Vert F-I_{2}\left(  F\right)  \right\Vert _{L^{2}\left(  A\left(
r_{1},r_{N}\right)  \right)  }^{2}=%
{\displaystyle\sum_{j=1}^{N-1}}
\left\Vert F-I_{2}\left(  F\right)  \right\Vert _{L^{2}\left(  A\left(
r_{j},r_{j+1}\right)  \right)  }^{2}.
\]
Now apply Theorem \ref{ThmMainSecondS2} for $p=q=2$ to the domain $A\left(
r_{j},r_{j+1}\right)  $ and we have
\[
\left\Vert F-I_{2}\left(  F\right)  \right\Vert _{L^{2}\left(  A\left(
r_{j},r_{j+1}\right)  \right)  }^{2}\leq C_{d}^{2}\ \cdot\left\Vert \Delta
F\right\Vert _{L^{2}\left(  A\left(  r_{j},r_{j+1}\right)  \right)  }^{2}.
\]
By summing up over $j=1,...,N$ and taking the square root we arrive at the
second statement.
\end{proof}

\section{Biharmonic interpolation splines}

In this section we want to provide the error estimate for biharmonic
interpolation splines on annular domains, see Theorem \ref{Thm1} in the
introduction. Note that a simple consequence of the error estimate is the
uniqueness of the biharmonic interpolation spline for a given data function
$F.$

We emphasize that the assumption for the data function, namely $F\in
C^{4}(\overline{A\left(  r_{1},r_{N}\right)  }),$ in Theorem \ref{Thm1} is
weaker than the usual assumption for proving the existence of a biharmonic
interpolation spline. Indeed, the first author has proved the existence of a
biharmonic spline interpolating a function $F$ in \cite[p. 446--453]%
{kounchevbook} with Dirichlet boundary conditions (\ref{eqnDirich}) using a
priori estimates for elliptic boundary values problems. For dimension $2$ it
is required that the data functions $F$ are from the fractional Sobolev space
$H^{7/2}\left(  \mathbb{T}\right)  \subset H^{3}\left(  \mathbb{T}\right)  $.
For dimension $2$, A. Bejancu has shown in \cite{Bejan13} the existence of a
biharmonic spline interpolating a function $F$ with Beppo-Levi boundary
conditions and a data function $F$ in the weighted Wiener algebra
$W^{2}\left(  \mathbb{T}\right)  $ (which is less restrictive) using Fourier
series techniques, see also \cite{Bejan11}.

Our approach to the proof for the error estimate of biharmonic interpolation
splines is inspired by the exposition of Carl de Boor in \cite{deBoor} for the
error estimate of cubic interpolation splines which is deduced via error
estimates with piecewise linear functions on $t_{1}<...<t_{N}$ in an iterative
way. The main observation in the intermediate step of the proof in
\cite{deBoor} is that the second derivative of interpolation cubic spline is
the best $L_{2}$ approximation to the second derivative of the interpolated
function. This fact depends on an orthogonality relation between the error and
linear splines. We shall use a similar argument in our context of biharmonic
and harmonic splines and for this we use the following well known fact from
best approximation in Hilbert spaces. For convenience of the reader we include
the short proof:

\begin{proposition}
\label{PropOrthostar}Let $U$ be a subspace of a Hilbert space $H,$ and $f\in
H$. Assume that $\varphi_{0}\in U$ has the property that
\begin{equation}
\left\langle f-\varphi_{0},\varphi\right\rangle =0\text{ for all }\varphi\in
U. \label{eqorthogonal}%
\end{equation}
Then $\left\Vert f-\varphi_{0}\right\Vert \leq\left\Vert f-\varphi\right\Vert
$ for all $\varphi\in U.$
\end{proposition}

\begin{proof}
Put $g=f-\varphi_{0}.$ Due to the orthogonality condition (\ref{eqorthogonal})
we have
\[
\left\Vert g-\varphi\right\Vert ^{2}=\left\Vert g\right\Vert ^{2}%
-2\left\langle g,\varphi\right\rangle +\left\Vert \varphi\right\Vert
^{2}=\left\Vert g\right\Vert ^{2}+\left\Vert \varphi\right\Vert ^{2}%
\geq\left\Vert g\right\Vert ^{2}.
\]
It follows that $\left\Vert g\right\Vert =\left\Vert f-\varphi_{0}\right\Vert
\leq\left\Vert f-\varphi_{0}-\varphi\right\Vert .$ Since $U$ is a subspace we
can replace $\varphi\in U$ by $\varphi_{0}+\varphi\in U$, and the proof is finished.
\end{proof}

Let us recall the main result of the paper stated in the introduction:

\begin{theorem}
Let $r_{1}<...<r_{N}$ and let $F\in C^{4}(\overline{A\left(  r_{1}%
,r_{N}\right)  })$. Assume that $I_{4}\left(  F\right)  $ is a biharmonic
spline for the partition $A\left(  r_{1},r_{2}\right)  ,....,A\left(
r_{N-1},r_{N}\right)  $ which satisfies the transfinite interpolation
conditions
\[
I_{4}\left(  x\right)  =F\left(  x\right)  \text{ for all }x\text{ with
}\left\vert x\right\vert =r_{j}\text{ and }j=1,...,N,
\]
and
\begin{equation}
\frac{\partial I_{4}\left(  F\right)  }{\partial n}\left(  x\right)
=\frac{\partial F\left(  x\right)  }{\partial n}\text{ for all }\left\vert
x\right\vert =r_{1}\text{ and for all }\left\vert x\right\vert =r_{N}.
\label{eqnDirich}%
\end{equation}
Then the following error estimate holds:
\[
\left\Vert F-I_{4}\left(  F\right)  \right\Vert _{L^{2}\left(  A\left(
r_{1},r_{N}\right)  \right)  }\leq\left(  C_{d}\right)  ^{2}\max
_{j=1,...,N-1}\left\vert r_{j+1}-r_{j}\right\vert ^{4}\left\Vert \Delta
^{2}F\right\Vert _{L^{2}\left(  A\left(  r_{1},r_{N}\right)  \right)  }.
\]
where $C_{d}=\max\left\{  \frac{1}{2d},\frac{1}{8}\right\}  .$
\end{theorem}

\begin{proof}
At first we note that
\[
\left\Vert F-I_{4}\left(  F\right)  \right\Vert _{L^{2}\left(  A\left(
r_{1},r_{N}\right)  \right)  }^{2}=%
{\displaystyle\sum_{j=1}^{N-1}}
\left\Vert F-I_{4}\left(  F\right)  \right\Vert _{L^{2}\left(  A\left(
r_{j},r_{j+1}\right)  \right)  }^{2}.
\]
We apply (\ref{eqineqL2}) to the function $f=F-I_{4}\left(  F\right)  $ on the
annular domain $A\left(  r_{j},r_{j+1}\right)  $: note that $f$ vanishes for
any $x$ with $\left\vert x\right\vert =r_{j}$ for $j=1,...,N,$ hence%
\[
\left\Vert F-I_{4}\left(  F\right)  \right\Vert _{L^{2}\left(  A\left(
r_{j},r_{j+1}\right)  \right)  }^{2}\leq C_{d}^{2}\left(  r_{j+1}%
-r_{j}\right)  ^{4}\left\Vert \Delta F-\Delta\left(  I_{4}\left(  F\right)
\right)  \right\Vert _{L^{2}\left(  A\left(  r_{j},r_{j+1}\right)  \right)
}^{2}.
\]
By summing up and taking the square root we see that
\[
\left\Vert F-I_{4}\left(  F\right)  \right\Vert _{L^{2}\left(  A\left(
r_{1},r_{N}\right)  \right)  }\leq C_{d}\max_{j=1,...,N-1}\left(
r_{j+1}-r_{j}\right)  ^{2}\cdot\left\Vert \Delta F-\Delta\left(  I_{4}\left(
F\right)  \right)  \right\Vert _{L^{2}\left(  A\left(  r_{1},r_{N}\right)
\right)  }.
\]
Now we want to estimate the right hand side: let us put $f_{0}=\Delta F$ and
$\varphi_{0}=\Delta\left(  I_{4}\left(  F\right)  \right)  .$ Note that
$\varphi_{0}$ is a harmonic spline -- but unfortunately it does not
interpolate the function $f_{0},$ so we cannot repeat the error estimate for
interpolating harmonic splines. In Theorem \ref{Thm3a} below we prove that the
equality
\[
\left\langle f_{0}-\varphi_{0},\varphi\right\rangle _{L^{2}\left(  A\left(
r_{1},r_{N}\right)  \right)  }=0
\]
holds for all harmonic splines $\varphi$ for the partition $A\left(
r_{1},r_{2}\right)  ,....,A\left(  r_{N-1},r_{N}\right)  $. Then Proposition
\ref{PropOrthostar} implies that
\[
\left\Vert \Delta F-\Delta\left(  I_{4}\left(  F\right)  \right)  \right\Vert
_{L^{2}\left(  A\left(  r_{1},r_{N}\right)  \right)  }=\left\Vert
f_{0}-\varphi_{0}\right\Vert _{L^{2}\left(  A\left(  r_{1},r_{N}\right)
\right)  }\leq\left\Vert f_{0}-\varphi\right\Vert _{L^{2}\left(  A\left(
r_{1},r_{N}\right)  \right)  }%
\]
holds for all harmonic splines $\varphi$ for the partition $A\left(
r_{1},r_{2}\right)  ,....,A\left(  r_{N-1},r_{N}\right)  .$ Let us take the
harmonic spline $\varphi:=I_{2}\left(  \Delta F\right)  .$ Hence,
\[
\left\Vert f_{0}-\varphi\right\Vert _{L^{2}\left(  A\left(  r_{1}%
,r_{N}\right)  \right)  }=\left\Vert \Delta F-I_{2}\left(  \Delta F\right)
\right\Vert _{L^{2}\left(  A\left(  r_{1},r_{N}\right)  \right)  }%
\]
is the error for the harmonic spline $\varphi$ interpolating $\Delta F,$ and
by Theorem \ref{ThmMainHarmSpAnn} we obtain
\[
\left\Vert \Delta F-I_{2}\left(  \Delta F\right)  \right\Vert _{L^{2}\left(
A\left(  r_{1},r_{N}\right)  \right)  }\leq C_{d}\max_{j=1,...,N-1}\left(
r_{j+1}-r_{j}\right)  ^{2}\left\Vert \Delta^{2}F\right\Vert _{L^{2}\left(
A\left(  r_{1},r_{N}\right)  \right)  }.
\]
It follows that
\[
\left\Vert F-I_{4}\left(  F\right)  \right\Vert _{L^{2}\left(  A\left(
r_{1},r_{N}\right)  \right)  }\leq C_{d}^{2}\max_{j=1,...,N-1}\left(
r_{j+1}-r_{j}\right)  ^{4}\cdot\left\Vert \Delta^{2}F\right\Vert
_{L^{2}\left(  A\left(  r_{1},r_{N}\right)  \right)  }.
\]
This ends the proof.
\end{proof}

\section{Proof of the orthogonality relation}

In this section we want to prove the following result:

\begin{theorem}
\label{Thm3a}Let $r_{1}<...<r_{N}$ and let $F\in C^{4}\left(  \overline
{\Omega}\right)  $ and $I_{4}\left(  F\right)  $ as in Theorem \ref{Thm1}.
Then for all harmonic splines $\varphi$ for the partition $A\left(
r_{j},r_{j+1}\right)  $ for $j=1,...,N$ $\mathbb{\ }$
\begin{equation}%
{\displaystyle\int_{A\left(  r_{1},r_{n}\right)  }}
\left(  \Delta F\left(  x\right)  -\Delta I_{4}\left(  F\right)  \left(
x\right)  \right)  \cdot\varphi\left(  x\right)  dx=0. \label{eqorthogo}%
\end{equation}

\end{theorem}

Let us write $f\left(  x\right)  =F\left(  x\right)  -I_{4}\left(  F\right)
\left(  x\right)  $. Then $f\in C^{4}(\overline{A\left(  r_{1},r_{N}\right)
})$ and $\varphi\in C^{2}\left(  A\left(  r_{1},r_{N}\right)  \right)  \cap
C(\overline{A\left(  r_{1},r_{N}\right)  }).$ Clearly we have
\[%
{\displaystyle\int_{A\left(  r_{1},r_{n}\right)  }}
\Delta f\left(  x\right)  \cdot\varphi\left(  x\right)  dx=%
{\displaystyle\sum_{j=1}^{n-1}}
{\displaystyle\int_{A\left(  r_{j},r_{j+1}\right)  }}
\Delta f\left(  x\right)  \cdot\varphi\left(  x\right)  dx.
\]
A natural approach is to prove (\ref{eqorthogo}) by appling Green's formula
(see \cite[p. 307]{ArGa}) to each summand on the right hand side. In order to
apply Green's formula, let us take $\rho_{j}<\rho_{j+1}$ in the interval
$\left(  r_{j},r_{j+1}\right)  $. Then $\varphi$ and $f$ are twice
differentiable in a neighborhood of $\overline{A\left(  \rho_{j},\rho
_{j+1}\right)  }\subset A\left(  r_{j},r_{j+1}\right)  $ and we apply Green's
formula. Since $\varphi$ is harmonic we obtain that
\begin{equation}%
{\displaystyle\int_{A\left(  \rho_{j},\rho_{j+1}\right)  }}
\Delta f\left(  x\right)  \cdot\varphi\left(  x\right)  dx=R\left(  \rho
_{j+1}\right)  -R\left(  \rho_{j}\right)  \label{eqGreen1}%
\end{equation}
where $R\left(  \rho\right)  $ is defined by
\[
R\left(  \rho\right)  =%
{\displaystyle\int_{_{S_{\rho}}}}
\frac{\partial f}{\partial n}\left(  y\right)  \varphi\left(  y\right)
d\sigma_{\rho}\left(  y\right)  -%
{\displaystyle\int_{_{S_{\rho}}}}
f\left(  y\right)  \frac{\partial\varphi}{\partial n}\left(  y\right)
d\sigma_{\rho}\left(  y\right)  ,
\]
and $\partial/\partial n$ denotes the exterior normal derivative, and
$\sigma_{\rho}$ is the surface measure on the sphere $S_{\rho}:=\left\{
x\in\mathbb{R}^{d}:\left\vert x\right\vert =\rho\right\}  $ for $j=1,2.$ Now
we want to take limits $\rho_{j}\rightarrow r_{j}$ and $\rho_{j+1}\rightarrow
r_{j+1}$ in (\ref{eqGreen1}). The limit for the left hand side clearly exists,
so
\[%
{\displaystyle\int_{A\left(  r_{1},r_{N}\right)  }}
\Delta f\left(  x\right)  \cdot\varphi\left(  x\right)  dx=%
{\displaystyle\sum_{j=1}^{N-1}}
\left(  \lim_{\rho_{j+1}<r_{j+1},\rho_{j+1}\rightarrow r_{j+1}}R\left(
\rho_{j+1}\right)  -\lim_{\rho_{j}>r_{j}\rho_{j}\rightarrow r_{j}}R\left(
r_{j}\right)  \right)  .
\]
We know that $f\left(  y\right)  $ vanishes for $\left\vert y\right\vert
=\rho_{j}$ but it seems to be unclear whether the expression
\[%
{\displaystyle\int_{_{S_{\rho}}}}
f\left(  y\right)  \frac{\partial\varphi}{\partial n}\left(  y\right)
d\sigma_{\rho}\left(  y\right)
\]
for $\rho\rightarrow r_{j}$ has a limit, and whether this limit is $0$ (as we
would expect) since the harmonic spline $\varphi\left(  y\right)  $ may have
an unbounded gradient for $y>r_{j}.$Thus this approach unfortunately does not
provide a proof of the statement, and we will pursue a proof of Theorem
\ref{Thm3a} using facts from the theory of spherical harmonics.

We shall write $x\in\mathbb{R}^{d}$ in spherical coordinates $x=r\theta$ with
$\theta\in\mathbb{S}^{d-1}=\left\{  x\in\mathbb{R}^{d}:\left\vert x\right\vert
=1\right\}  .$ Let $d\theta$ be the surface measure of $\mathbb{S}^{d-1}$ and
define the inner product
\begin{equation}
\left\langle f,g\right\rangle _{\mathbb{S}^{d-1}}:=\int_{\mathbb{S}^{d-1}%
}f\left(  \theta\right)  \overline{g\left(  \theta\right)  }d\theta.
\label{eqinnerpr}%
\end{equation}
Let $\mathcal{H}_{k}\left(  \mathbb{R}^{d}\right)  $ be the set of all
harmonic homogeneous complex-valued polynomials of degree $k.$ Then
$f\in\mathcal{H}_{k}\left(  \mathbb{R}^{d}\right)  $ is called a \emph{solid
harmonic} and the restriction of $f$ to $\mathbb{S}^{d-1}$ a \emph{spherical
harmonic} of degree $k$ and we set $a_{k}:=\dim\mathcal{H}_{k}\left(
\mathbb{R}^{d}\right)  ,$ see \cite{steinWeiss}, \cite{kounchevbook} for details.

Assume that $Y_{k,\ell}:\mathbb{R}^{d}\rightarrow\mathbb{R},\ell=1,...,a_{k},$
is an \emph{orthonormal basis} of $\mathcal{H}_{k}\left(  \mathbb{R}%
^{d}\right)  $ with respect to (\ref{eqinnerpr}). Since $Y_{k,\ell}$ is
homogeneous of degree $k$ we have $Y_{k,\ell}\left(  x\right)  =r^{k}Y_{k\ell
}\left(  \theta\right)  $ for $x=r\theta.$ For a continuous function
$f:\overline{A\left(  r_{1},r_{N}\right)  }\rightarrow\mathbb{C}$ we define
the \emph{Fourier-Laplace coefficient }$f_{k,\ell}\left(  r\right)  $%
\emph{\ }for $r\in\left[  r_{1},r_{N}\right]  $ by
\begin{equation}
f_{k,\ell}\left(  r\right)  =\int_{\mathbb{S}^{d-1}}f\left(  r\theta\right)
Y_{k,\ell}\left(  \theta\right)  d\theta. \label{eqfourier}%
\end{equation}
The \emph{Fourier-Laplace series }of $f:A\left(  a,b\right)  \rightarrow
\mathbb{C}$ is defined by the formal expansion
\begin{equation}
f\left(  r\theta\right)  =\sum_{k=0}^{\infty}\sum_{\ell=1}^{a_{k}}f_{k,\ell
}\left(  r\right)  Y_{k,\ell}\left(  \theta\right)  . \label{frtheta}%
\end{equation}
If $\theta\longmapsto f\left(  r\theta\right)  $ is a continuous, it is a
function in $L^{2}\left(  \mathbb{S}^{d-1}\right)  $ and since $\left(
Y_{k,\ell}\left(  \theta\right)  \right)  _{k\in\mathbb{N}_{0},l=1,...a_{k}}$
is complete orthonormal basis one has
\[
\int_{\mathbb{S}^{d-1}}\left\vert f\left(  r\theta\right)  \right\vert
^{2}d\theta=\sum_{k=0}^{\infty}\sum_{\ell=1}^{a_{k}}\left\vert f_{k,\ell
}\left(  r\right)  \right\vert ^{2}.
\]
If $f,g$ are continuous on $\overline{A\left(  r_{1},r_{N}\right)  }$ with
Fourier-Laplace coeffients $f_{k,\ell}\left(  r\right)  $ and $g_{k,\ell
}\left(  r\right)  $ respectively we obtain
\[
\int_{\mathbb{S}^{d-1}}f\left(  r\theta\right)  g\left(  r\theta\right)
d\theta=\sum_{k=0}^{\infty}\sum_{\ell=1}^{a_{k}}f_{k,\ell}\left(  r\right)
g_{k,\ell}\left(  r\right)
\]
and the series on the right hand side converges absolutely. Multiply this
equation with $r^{d-1}$ and integrate with respect to $dr,$ hence we have
\begin{align}
\int_{A\left(  r_{1},r_{N}\right)  }f\left(  x\right)  g\left(  x\right)  dx
&  =\int_{r_{1}}^{r_{N}}\int_{\mathbb{S}^{d-1}}f\left(  r\theta\right)
g\left(  r\theta\right)  r^{d-1}d\theta dr\label{eqscalarproduct}\\
&  =\sum_{k=0}^{\infty}\sum_{\ell=1}^{a_{k}}\int_{r_{1}}^{r_{N}}f_{k,\ell
}\left(  r\right)  g_{k,\ell}\left(  r\right)  r^{d-1}dr.
\label{eqscalarproduct2}%
\end{align}
Let us define the univariate differential operators
\[
L_{k}\left(  f\right)  =\frac{\partial^{2}}{\partial^{2}r}f+\frac{d-1}{r}%
\frac{\partial}{\partial r}f-\frac{k\left(  k+d-2\right)  }{r^{2}}f.
\]
The following result is now crucial for our arguments. Since we could not find
a reference for this result we include a proof (in \cite{kounchevbook} it is
proved under the stronger assumption that $\Delta f$ has a absolutely
convergent Fourier-Laplace series).

\begin{theorem}
\label{ThmLk}Let $a<b$ be real numbers and assume that $f:A\left(  a,b\right)
\rightarrow\mathbb{C}$ is continuously partially differentiable of order $2.$
Then the Fourier-Laplace coefficient $f_{k,\ell}$ of $f$ is twice
differentiable on $\left(  a,b\right)  $ and%
\[
\int_{\mathbb{S}^{d-1}}\left(  \Delta f\right)  \left(  r\theta\right)  \cdot
Y_{k,\ell}\left(  \theta\right)  d\theta=L_{k}\left(  f_{k,\ell}\right)
\left(  r\right)  .
\]
Thus the $\left(  k,\ell\right)  $-th Fourier Laplace coffiecient of $\Delta
f$ is equal to $L_{k}\left(  f_{k,\ell}\right)  .$
\end{theorem}

\begin{proof}
It is easy to see that $f_{k,\ell}$ is twice differentiable on $\left(
a,b\right)  $. Assume now that $\varphi\left(  x\right)  =Y_{k,\ell}\left(
x\right)  $. Since $Y_{k,\ell}\left(  r\theta\right)  =r^{k}Y_{k,\ell}\left(
\theta\right)  $ for all $\theta\in\mathbb{S}^{d-1},$ it follows that
\begin{align*}
R_{r}  &  :=%
{\displaystyle\int_{_{S_{r}}}}
\frac{\partial}{\partial r}f\left(  r\theta\right)  \cdot r^{k}Y_{k,\ell
}\left(  \theta\right)  d\sigma_{r}\left(  y\right)  -%
{\displaystyle\int_{_{S_{r}}}}
f\left(  r\theta\right)  \frac{\partial}{\partial r}\left(  r^{k}Y_{k,\ell
}\left(  \theta\right)  \right)  d\sigma_{r}\left(  y\right) \\
&  =r^{k+d-1}\frac{\partial}{\partial r}\int_{\mathbb{S}^{d-1}}f\left(
r\theta\right)  Y_{k,\ell}\left(  \theta\right)  d\sigma-r^{d-1}%
\int_{\mathbb{S}^{d-1}}f\left(  r\theta\right)  \frac{\partial}{\partial
r}\left(  r^{k}Y_{k,\ell}\left(  \theta\right)  \right)  d\sigma\\
&  =r^{k+d-1}\frac{\partial}{\partial r}f_{k,\ell}\left(  r\right)
-kr^{k+d-2}f_{k,\ell}\left(  r\right)  =:F\left(  r\right)  .
\end{align*}
Let us now take $r_{1}=r$ and $r_{2}=r+h$, then by Green's formula (see above)
we have
\[
\frac{1}{h}%
{\displaystyle\int_{A\left(  r,r+h\right)  }}
\Delta f\left(  x\right)  \cdot Y_{k,\ell}\left(  x\right)  dx=\frac{F\left(
r+h\right)  -F\left(  r\right)  }{h}.
\]
We can take the limit on the right hand side and we see that
\[
F^{\prime}\left(  r\right)  =r^{k+d-1}\frac{\partial}{\partial r}f_{k,\ell
}\left(  r\right)  +\left(  d-1\right)  r^{k+d-2}\frac{\partial}{\partial
r}f_{k,\ell}\left(  r\right)  -k\left(  k+d-2\right)  r^{k+d-3}f_{k,\ell
}\left(  r\right)  .
\]
On the other hand,
\[
\frac{1}{h}%
{\displaystyle\int_{A\left(  r,r+h\right)  }}
\Delta f\left(  x\right)  \cdot Y_{k,\ell}\left(  x\right)  dx=\frac{1}{h}%
\int_{r}^{r+h}\left(  \int_{\mathbb{S}^{d-1}}\Delta f\left(  r\theta\right)
Y_{k,\ell}\left(  \theta\right)  d\sigma\right)  r^{k+d-1}dr
\]
converges to
\[
\left(  \int_{\mathbb{S}^{d-1}}\Delta f\left(  r\theta\right)  Y_{k,\ell
}\left(  \theta\right)  d\sigma\right)  r^{k+d-1}.
\]
The proof is complete.
\end{proof}

\textbf{Proof of Theorem \ref{Thm3a}}:

\begin{proof}
Let $f:=F-I_{4}F\in C^{2}(\overline{A\left(  r_{1},r_{N}\right)  })$ and
$\varphi$ a harmonic spline. Let $f_{k,\ell}\left(  r\right)  $ be the
$\left(  k,\ell\right)  $-th Fourier-Laplace coefficient of $f,$ so
\[
f_{k,\ell}\left(  r\right)  =\int_{\mathbb{S}^{d-1}}f\left(  r_{j}%
\theta\right)  Y_{k,\ell}\left(  \theta\right)  d\theta.
\]
Then $f_{k,\ell}$ is twice differentiable and $f_{k,\ell}\left(  r_{j}\right)
=0$ for $j=1,...,N.$ Define $g:=\Delta f,$ and let $g_{k,\ell}$ and
$\varphi_{k,\ell}$ be the Fourier-Laplace coefficients of $g$ and $\varphi$
respectively. By formula (\ref{eqscalarproduct}) applied to $\Delta f$ and
$\varphi$ it suffices to show that
\[
I_{k,,\ell}:=\int_{r_{1}}^{r_{N}}g_{k,\ell}\left(  r\right)  \overline
{\varphi_{k,\ell}\left(  r\right)  }r^{d-1}dr=0.
\]
Theorem \ref{ThmLk} applied to $f$ on the domain $A\left(  r_{j}%
,r_{j+1}\right)  $ shows that $g_{k,\ell}\left(  r\right)  =L_{k}f_{k,\ell
}\left(  r\right)  $ for $r\in\left(  r_{j},r_{j+1}\right)  .$ Thus we see
that
\[
I_{k,,\ell}=\int_{r_{j}}^{r_{j+1}}L_{k}f_{k,\ell}\left(  r\right)
\overline{\varphi_{k,\ell}\left(  r\right)  }r^{d-1}dr.
\]
Partial integration shows that with $h_{k,\ell}\left(  r\right)
=\overline{\varphi_{k,\ell}\left(  r\right)  }r^{d-1}$ shows that
\[
\int_{r_{j}}^{r_{j+1}}\frac{\partial^{2}f_{k,\ell}\left(  r\right)  }%
{\partial^{2}r}h_{k,\ell}\left(  r\right)  dr=\frac{\partial f_{k,\ell}\left(
r\right)  }{\partial r}h_{k,\ell}\left(  r\right)  \mid_{r_{j}}^{r_{j+1}%
}-f_{k,\ell}\left(  r\right)  \frac{\partial}{\partial r}h_{k,\ell}\left(
r\right)  \mid_{r_{j}}^{r_{j+1}}+\int_{r_{j}}^{r_{j+1}}f_{k,\ell}\left(
r\right)  \frac{\partial^{2}h_{k,\ell}}{\partial^{2}r}dr.
\]
Since $f_{k,\ell}\left(  r_{j}\right)  =0$ and $h_{k,\ell}\left(  r\right)  $
is analytic function for $r>0,$ the middle term vanishes. Next we see that
\begin{align*}
&  \int_{r_{j}}^{r_{j+1}}\frac{d-1}{r}\frac{\partial}{\partial r}f_{k,\ell
}\left(  r\right)  \cdot h_{k,\ell}\left(  r\right)  dr\\
&  =f_{k,\ell}\left(  r\right)  \frac{d-1}{r}h_{k,\ell}\left(  r\right)
\mid_{r_{j}}^{r_{j+1}}-\left(  d-1\right)  \int_{r_{j}}^{r_{j+1}}f_{k,\ell
}\left(  r\right)  \frac{\partial}{\partial r}\frac{h_{k,\ell}\left(
r\right)  }{r}dr.
\end{align*}
It follows that
\[
\int_{r_{j}}^{r_{j+1}}L_{k}f_{k,\ell}\left(  r\right)  h_{k,\ell}\left(
r\right)  dr=\frac{\partial f_{k,\ell}\left(  r\right)  }{\partial r}%
h_{k,\ell}\left(  r\right)  \mid_{r_{j}}^{r_{j+1}}+\int_{r_{j}}^{r_{j+1}%
}f_{k,\ell}\left(  r\right)  M\left(  h_{k,\ell}\right)  \left(  r\right)  dr
\]
and
\begin{align*}
&  M\left(  h_{k,\ell}\right)  \left(  r\right) \\
&  =\frac{\partial^{2}h_{k,\ell}}{\partial^{2}r}-\frac{\partial}{\partial
r}\frac{\left(  d-1\right)  h_{k,\ell}\left(  r\right)  }{r}-\frac{k\left(
k+d-2\right)  }{r^{2}}h_{k,\ell}\left(  r\right) \\
&  =r^{d-1}\frac{\partial^{2}}{\partial^{2}r}\overline{\varphi_{k,\ell}\left(
r\right)  }+2\left(  d-1\right)  r^{d-2}\frac{\partial}{\partial r}%
\overline{\varphi_{k,\ell}\left(  r\right)  }+\overline{\varphi_{k,\ell
}\left(  r\right)  }\left(  d-1\right)  \left(  d-2\right)  r^{d-3}\\
&  -\left(  d-1\right)  \left(  d-2\right)  r^{d-3}\overline{\varphi_{k,\ell
}\left(  r\right)  }-\left(  d-1\right)  r^{d-2}\frac{\partial}{\partial
r}\overline{\varphi_{k,\ell}\left(  r\right)  }-\frac{k\left(  k+d-2\right)
}{r^{2}}\overline{\varphi_{k,\ell}\left(  r\right)  }r^{d-1}%
\end{align*}
and we see that
\[
M\left(  h_{k,\ell}\right)  \left(  r\right)  =r^{d-1}\left(  \frac
{\partial^{2}}{\partial^{2}r}\overline{\varphi_{k,\ell}\left(  r\right)
}+\frac{\left(  d-1\right)  }{r}\frac{\partial}{\partial r}\overline
{\varphi_{k,\ell}\left(  r\right)  }-\frac{k\left(  k+d-2\right)  }{r^{2}%
}\overline{\varphi_{k,\ell}\left(  r\right)  }\right)  .
\]
Theorem \ref{ThmLk} applied to the harmonic function $\varphi$ on the domain
$A\left(  r_{j},r_{j+1}\right)  $ shows that $L_{k}\varphi_{k,\ell}\left(
r\right)  =0$ for $r\in\left(  r_{j},r_{j+1}\right)  $ and therefore $M\left(
h_{k,\ell}\right)  \left(  r\right)  =0$ for $r\in\left(  r_{j},r_{j+1}%
\right)  .$ It follows that
\begin{align*}%
{\displaystyle\sum_{j=1}^{N}}
\int_{r_{j}}^{r_{j+1}}L_{k}f_{k,\ell}\left(  r\right)  h_{k,\ell}\left(
r\right)  dr  &  =%
{\displaystyle\sum_{j=1}^{N}}
\frac{\partial f_{k,\ell}\left(  r\right)  }{\partial r}h_{k,\ell}\left(
r\right)  \mid_{r_{j}}^{r_{j+1}}\\
&  =\frac{\partial f_{k,\ell}\left(  r_{N}\right)  }{\partial r}h_{k,\ell
}\left(  r_{N}\right)  -\frac{\partial f_{k,\ell}\left(  r_{1}\right)
}{\partial r}h_{k,\ell}\left(  r_{1}\right)  .
\end{align*}
Further for $r^{\ast}=r_{N}$ or $r^{\ast}=r_{1}$ we have
\begin{align*}
\frac{\partial f_{k,\ell}\left(  r^{\ast}\right)  }{\partial r}h_{k,\ell
}\left(  r^{\ast}\right)   &  =\lim_{r\rightarrow r^{\ast}}\frac{\partial
}{\partial r}\int_{\mathbb{S}^{d-1}}f\left(  r\theta\right)  Y_{k,\ell}\left(
\theta\right)  d\theta\cdot h_{k,\ell}\left(  r\right) \\
&  =\lim_{r\rightarrow r^{\ast}}\int_{\mathbb{S}^{d-1}}\frac{\partial
}{\partial r}f\left(  r\theta\right)  Y_{k,\ell}\left(  \theta\right)
d\theta\cdot h_{k,\ell}\left(  r\right) \\
&  =\int_{\mathbb{S}^{d-1}}\frac{\partial}{\partial r}f\left(  r^{\ast}%
\theta\right)  Y_{k,\ell}\left(  \theta\right)  d\theta\cdot h_{k,\ell}\left(
r^{\ast}\right)  =0
\end{align*}
since $f$ has normal derivative $0$ at $r^{\ast}\theta.$
\end{proof}

\section{The constant $c_{d}\left(  \Omega\right)  $ for the annulus}

In this section we want to determine the constant $c_{d}\left(  A\left(
r,R\right)  \right)  $ for the annulus. The Green function for the annulus is
known, for a nice exposition see \cite{GrVu15}. On the other hand, Theorem
\ref{Remark1} describes a simpler way to solve the problem for $d\geq2$. We
define for $d\geq3$ the harmonic function $h_{d}\left(  x\right)  =\left(
\frac{\left\vert x\right\vert }{R}\right)  ^{2-d}-1$ and $h_{2}\left(
x\right)  =\log\left(  \frac{\left\vert x\right\vert }{R}\right)  $ for $d=2.$
Then
\[
T_{0}\left(  x\right)  =\frac{R^{2}}{2d}\left(  1-\frac{\left\vert
x\right\vert ^{2}}{R^{2}}-\frac{1-\frac{r^{2}}{R^{2}}}{h_{d}\left(  r\right)
}h_{d}\left(  x\right)  \right)  .
\]
has the property that it vanishes for $\left\vert x\right\vert =R$ and
$\left\vert x\right\vert =r$ and satisfies $\Delta f=-1.$

\begin{theorem}
\label{ThmAnnulus}Let $A\left(  r,R\right)  $ be the annulus in $\mathbb{R}%
^{d}$ for $d\geq2$ and set $D=\left(  d-2\right)  /2.$ Then
\begin{equation}
c_{d}\left(  A\left(  r,R\right)  \right)  =\frac{\left(  R-r\right)  ^{2}%
}{2d}H_{d}\left(  \frac{r}{R}\right)  \label{eqCformula}%
\end{equation}
where $H_{d}$ is defined for $d\geq3$ by
\begin{equation}
H_{d}\left(  \rho\right)  =\frac{1}{\left(  1-\rho\right)  ^{2}}\left(
1+\rho^{2D}\frac{1-\rho^{2}}{1-\rho^{2D}}-\frac{D+1}{D}\left(  D\rho^{2D}%
\frac{1-\rho^{2}}{1-\rho^{2D}}\right)  ^{\frac{1}{D+1}}\right)  \label{eqHd}%
\end{equation}
and for $d=2$
\begin{equation}
H_{2}\left(  \rho\right)  =\frac{1}{\left(  1-\rho\right)  ^{2}}\left(
1+\frac{1-\rho^{2}}{2\log\rho}-\frac{1}{2}\frac{1-\rho^{2}}{\log\rho}%
\log\left(  \frac{1-\rho^{2}}{-2\log\rho}\right)  \right)  .\label{eqHd2}%
\end{equation}

\end{theorem}

\begin{proof}
At first we assume that $d>2.$ Then $h_{d}\left(  x\right)  =\left(
\frac{\left\vert x\right\vert }{R}\right)  ^{2-d}-1.$ Put $u=\left\vert
x\right\vert ^{2}/R^{2}$ and $\rho=r/R<1$. Then
\[
\ T_{0}\left(  x\right)  =\frac{R^{2}}{2d}\left[  1-u-\frac{1-\rho^{2}}%
{\rho^{-2D}-1}\left(  \left(  u\right)  ^{-D}-1\right)  \right]  .
\]
It follows that
\begin{equation}
c_{d}\left(  A\left(  r,R\right)  \right)  =\sup_{x\in A\left(  r,R\right)
}\ T_{0}\left(  x\right)  =\frac{R^{2}}{2d}\sup_{\frac{r^{2}}{R^{2}}\leq
u\leq1}f_{d}\left(  u\right)  \label{eqcArR}%
\end{equation}
where we define
\[
f_{d}\left(  u\right)  =1-u-B_{d}\left(  \rho\right)  \left(  u^{-D}-1\right)
=1-u-B_{d}\left(  \rho\right)  u^{-D}+B_{d}\left(  \rho\right)
\]
and
\begin{equation}
B_{d}\left(  \rho\right)  =\frac{1-\rho^{2}}{\rho^{-2D}-1}=\rho^{2D}%
\frac{1-\rho^{2}}{1-\rho^{2D}}.\label{eqBrho}%
\end{equation}
Obviously $f_{d}\left(  1\right)  =0$ and $f_{d}\left(  \frac{r^{2}}{R^{2}%
}\right)  =f_{d}\left(  \rho^{2}\right)  =0.$ Hence the maximum of
$f_{d}\left(  u\right)  $ on the interval $\left[  \frac{r^{2}}{R^{2}%
},1\right]  $ is attained in the interior. Further
\[
\frac{d}{du}f_{d}\left(  u\right)  =-1+DB_{d}\left(  \rho\right)  u^{-D-1}.
\]
Hence $f_{d}^{\prime}\left(  u_{0}\right)  =0$ implies that $u_{0}%
^{D+1}=DB_{d}\left(  \rho\right)  .$ It follows that with $B=B_{d}\left(
\rho\right)  $ and inserting $u_{0}=\left(  DB\right)  ^{\frac{1}{D+1}}$
\begin{equation}
\sup_{\rho=\frac{r^{2}}{R^{2}}\leq u\leq1}f_{d}\left(  u\right)  =1-\left(
DB\right)  ^{\frac{1}{D+1}}-B\left(  \left(  DB\right)  ^{-\frac{D}{D+1}%
}-1\right)  =1+B-\frac{D+1}{D}\left(  DB\right)  ^{\frac{1}{D+1}}.\label{eqfd}%
\end{equation}
The identity (\ref{eqHd}) follows from (\ref{eqcArR}) and (\ref{eqfd}) by
writing
\[
\frac{R^{2}}{2d}=\frac{\left(  R-r\right)  ^{2}}{2d}\frac{1}{\left(
1-\rho\right)  ^{2}}.
\]
For $d=2$ we have $h_{2}\left(  x\right)  =\log\left(  \frac{\left\vert
x\right\vert }{R}\right)  =\frac{1}{2}\log\left(  \frac{\left\vert
x\right\vert }{R}\right)  ^{2}$ and with $u=\left\vert x\right\vert ^{2}%
/R^{2}$
\[
\ T_{0}\left(  x\right)  =\frac{R^{2}}{2d}\left[  1-u-\frac{1}{2}\frac
{1-\rho^{2}}{\log\rho}\log u\right]  .
\]
It follows that
\[
c_{d}\left(  A\left(  r,R\right)  \right)  =\sup_{x\in A\left(  r,R\right)
}\ T_{0}\left(  x\right)  =\frac{R^{2}}{2d}\sup_{\frac{r^{2}}{R^{2}}\leq
u\leq1}f_{2}\left(  u\right)
\]
where we define $f_{2}\left(  u\right)  =1-u-\frac{1}{2}\frac{1-\rho^{2}}%
{\log\rho}\log u$. Then $f_{2}\left(  \rho^{2}\right)  =0$ and $f_{1}\left(
1\right)  =0$ and $f_{2}^{\prime}\left(  u\right)  =-1-\frac{1}{2}\frac
{1-\rho^{2}}{\log\rho}\frac{1}{u}.$ Then $u=\frac{1-\rho^{2}}{-2\log\rho}$ is
the critical point and
\[
\sup_{\rho=\frac{r^{2}}{R^{2}}\leq u\leq1}f_{2}\left(  u\right)
=1+\frac{1-\rho^{2}}{2\log\rho}-\frac{1}{2}\frac{1-\rho^{2}}{\log\rho}%
\log\left(  \frac{1-\rho^{2}}{-2\log\rho}\right)  .
\]

\end{proof}

\begin{proposition}
For dimension $d=3$ the function $H_{3}\left(  \rho\right)  $ in Theorem
\ref{ThmAnnulus} is strictly decreasing on $\left[  0,1\right]  $ and
\[
\frac{\left(  R-r\right)  ^{2}}{8}\leq c_{d}\left(  A\left(  r,R\right)
\right)  \leq\frac{\left(  R-r\right)  ^{2}}{6}.
\]

\end{proposition}

\begin{proof}
For $d=3$ we have $D=\frac{1}{2}$ and
\begin{align*}
H_{3}\left(  \rho\right)   &  =\frac{1+\rho\frac{1-\rho^{2}}{1-\rho}-3\left(
\frac{1}{2}\rho\frac{1-\rho^{2}}{1-\rho}\right)  ^{\frac{2}{3}}}{\left(
1-\rho\right)  ^{2}}\\
&  =\frac{1+\rho\left(  1+\rho\right)  -3\left(  \frac{1}{2}\rho\left(
1+\rho\right)  \right)  ^{\frac{2}{3}}}{\left(  1-\rho\right)  ^{2}}.
\end{align*}
A computation shows that
\[
H_{3}^{\prime}\left(  \rho\right)  =\frac{\left(  -1\right)  \sqrt[3]%
{2}\left(  \rho\left(  \rho+1\right)  \right)  ^{\frac{2}{3}}}{\rho\left(
1-\rho\right)  ^{3}\left(  \rho+1\right)  }w\left(  \rho\right)
\]
where $w\left(  \rho\right)  =\rho^{2}+4\rho+1-3\left(  1+\rho\right)
\sqrt[3]{\frac{1}{2}\rho\left(  \rho+1\right)  }.$ Then $w\left(  \rho\right)
$ is non-negative since
\[
\left(  \rho^{2}+4\rho+1\right)  ^{3}-3^{3}\left(  1+\rho\right)  ^{3}\frac
{1}{2}\rho\left(  \rho+1\right)  =\frac{1}{2}\left(  \rho-1\right)
^{4}\left(  2\rho^{2}+5\rho+2\right)  \geq0.
\]
It follows that $H_{3}^{\prime}\left(  \rho\right)  <0$ for $\rho\in\left(
0,1\right)  $ and $H_{3}$ is strictly decreasing. Hence%
\begin{align*}
\frac{3}{4}  &  =\lim_{\rho\rightarrow1}\frac{1+\rho\left(  1+\rho\right)
-3\left(  \frac{1}{2}\rho\left(  1+\rho\right)  \right)  ^{\frac{2}{3}}%
}{\left(  1-\rho\right)  ^{2}}\leq H_{3}\left(  \rho\right) \\
&  \leq H_{3}\left(  0\right)  =1.
\end{align*}
The proof is complete.
\end{proof}

\begin{proposition}
For dimension $d=2$ the function $H_{2}\left(  \rho\right)  $ in Theorem
\ref{ThmAnnulus} is strictly decreasing on $\left[  0,1\right]  $ and
\[
\frac{\left(  R-r\right)  ^{2}}{8}\leq c_{d}\left(  A\left(  r,R\right)
\right)  \leq\frac{\left(  R-r\right)  ^{2}}{4}.
\]

\end{proposition}

\begin{proof}
Similarly one can see that the function $H_{2}\left(  \rho\right)  $ is
decreasing and
\[
\lim_{\rho\rightarrow0}H_{2}\left(  \rho\right)  =\lim_{\rho\rightarrow0}%
\frac{1+\frac{1-\rho^{2}}{2\log\rho}-\frac{1}{2}\frac{1-\rho^{2}}{\log\rho
}\log\left(  \frac{1-\rho^{2}}{-2\log\rho}\right)  }{\left(  1-\rho\right)
^{2}}=1
\]
and $\lim_{\rho\rightarrow1}H_{2}\left(  \rho\right)  =\frac{1}{2}.$ The proof
is complete.
\end{proof}

The discussion for $d\geq4$ is more technical and we need the following result:

\begin{lemma}
\label{LemMain}Let $d>2$, i.e. that $D=\left(  d-2\right)  /2>0.$ Then the
function
\[
B_{d}\left(  \rho\right)  =\rho^{2D}\frac{1-\rho^{2}}{1-\rho^{2D}}\text{ for
}\rho\in\left(  0,1\right)
\]
is increasing and positive on $\left(  0,1\right)  $ and the following limits
exist:
\[
\lim_{\rho\rightarrow1}B_{d}\left(  \rho\right)  =\frac{1}{D}\text{ and }%
\lim_{\rho\rightarrow1}B_{D}^{\prime}\left(  \rho\right)  =\frac{D+1}{D}\text{
and }\lim_{\rho\rightarrow1}B_{d}^{\prime\prime\prime}\left(  \rho\right)  =0.
\]
Further $B_{d}^{\prime\prime\prime}$ is positive on $\left(  0,1\right)  $ for
$D>1,$ and negative for $D\in\left(  \frac{1}{2},1\right)  .$ Further
\begin{equation}
\frac{B_{d}^{\prime}\left(  \rho\right)  }{B_{d}\left(  \rho\right)  }%
=\frac{2\left(  D-\rho^{2}+\rho^{2D+2}-\rho^{2}D\right)  }{\rho\left(
1-\rho^{2D}\right)  \left(  1-\rho^{2}\right)  }.\label{eqBdnew}%
\end{equation}

\end{lemma}

\begin{proof}
For $D>0$ we see use the rule of l'Hospital $\lim_{\rho\rightarrow1}%
B_{d}\left(  \rho\right)  =\lim_{\rho\rightarrow1}\rho^{2D}\frac{1-\rho^{2}%
}{1-\rho^{2D}}=\frac{1}{D}.$ For $\rho\in\left(  0,1\right)  $ we have
\[
B_{D}^{\prime}\left(  \rho\right)  =\frac{2\rho^{2D-1}}{\left(  1-\rho
^{2D}\right)  ^{2}}\left(  D-\rho^{2}+\rho^{2D+2}-\rho^{2}D\right)
\rightarrow\frac{D+1}{D}%
\]
for $\rho\rightarrow1.$ Further we obtain equation (\ref{eqBdnew}) from the
last formula. Consider the function $g\left(  x\right)  =D-x+x^{D+1}-xD$,
then
\[
g^{\prime}\left(  x\right)  =\left(  D+1\right)  x^{D}-\left(  D+1\right)
=\left(  D+1\right)  \left(  x^{D}-1\right)  <0
\]
for $0<x<1$ for $D>0.$ Hence $g$ is strictly decreasing on $\left[
0,1\right]  $ and $g\left(  1\right)  =0,$ so $g\left(  x\right)  \geq
g\left(  1\right)  =0$ for all $x\in\left(  0,1\right)  .$ Thus $g$ is
positive in $\left(  0,1\right)  $ and it follows that $B_{d}\left(
\rho\right)  $ is strictly increasing for any $D>0.\ $Further
\[
\frac{d^{3}}{d\rho^{3}}\left(  \rho^{2D}\frac{1-\rho^{2}}{1-\rho^{2D}}\right)
=\frac{4D\rho^{2D-3}}{\left(  1-\rho^{2D}\right)  ^{4}}f\left(  \rho
^{2}\right)
\]
where
\begin{align*}
f\left(  x\right)   &  =\left(  2D-1\right)  \left(  D-1\right)  -x\left(
D+1\right)  \left(  2D+1\right)  +2x^{D}\left(  4D^{2}-1\right)  \allowbreak\\
&  -2x^{D+1}\left(  4D^{2}-1\right)  +\allowbreak x^{2D}\left(  2D^{2}%
+3D+1\right)  -x^{2D+1}\left(  2D^{2}-3D+1\right)  .
\end{align*}
One can verify that $f\left(  x\right)  $ has a zero of order $5$ at $x=1.$
Since $f$ is in in the linear space generated by the basis functions
$1,x,x^{D},x^{D+1},x^{2D},x^{2D+1}$ the function $f$ has at exactly $5$ zeros
at $x=1$ and no more positive zeros. It follows that $f$ is either positive or
negative on $\left(  0,1\right)  ,$ hence $f$ has the same sign as $f\left(
0\right)  =\left(  2D-1\right)  \left(  D-1\right)  .$ Since $f$ has a zero of
order $5$ at $x=1$ we see that
\[
\lim_{\rho\rightarrow1}B_{d}^{\prime\prime\prime}\left(  \rho\right)  =0.
\]
The proof is complete.
\end{proof}

\begin{theorem}
For dimension $d=4$ the function $H_{4}\left(  \rho\right)  $ in Theorem
\ref{ThmAnnulus} is constant and
\[
c_{d}\left(  A\left(  r,R\right)  \right)  =\frac{\left(  R-r\right)  ^{2}}%
{8}.
\]
For dimension $d>4$ the function $H_{d}$ in Theorem \ref{ThmAnnulus} is
strictly increasing on $\left[  0,1\right]  $ and
\begin{equation}
\frac{\left(  R-r\right)  ^{2}}{2d}\leq c_{d}\left(  A\left(  r,R\right)
\right)  \leq\frac{\left(  R-r\right)  ^{2}}{8}.\label{qcestimate8}%
\end{equation}

\end{theorem}

\begin{proof}
1. For $d=4$ we have $D=\left(  d-2\right)  /2=1$ and
\[
H_{4}\left(  \rho\right)  =\frac{1+\rho^{2}\frac{1-\rho^{2}}{1-\rho^{2}%
}-2\left(  \rho^{2}\frac{1-\rho^{2}}{1-\rho^{2}}\right)  ^{\frac{1}{2}}%
}{\left(  1-\rho\right)  ^{2}}=\frac{\rho^{2}+1-2\rho}{\left(  1-\rho\right)
^{2}}=1.
\]
2. Assume now $d>4$, so $D>1.$We analyse the function
\[
H_{d}\left(  \rho\right)  =\frac{G_{d}\left(  \rho\right)  }{\left(
1-\rho\right)  ^{2}}\text{ with }G_{d}\left(  \rho\right)  :=1+B_{d}\left(
\rho\right)  -\frac{D+1}{D}\left(  DB_{d}\left(  \rho\right)  \right)
^{\frac{1}{D+1}}.
\]
For $\rho\in\left(  0,1\right)  $ we compute
\[
H_{d}^{\prime}\left(  \rho\right)  =\frac{G_{d}^{\prime}\left(  \rho\right)
\left(  1-\rho\right)  +2G_{d}\left(  \rho\right)  }{1-\rho}.
\]
Consider the numerator of $H_{d}^{\prime}\left(  \rho\right)  $ defined by
\begin{equation}
N_{d}\left(  \rho\right)  :=G_{d}^{\prime}\left(  \rho\right)  \left(
1-\rho\right)  +2G_{d}\left(  \rho\right)  . \label{eqDefN}%
\end{equation}
For $D>1$ we will show that $N_{d}\left(  \rho\right)  $ is positive on
$\left(  0,1\right)  .$\ This clearly implies that $H_{d}$ is increasing for
$D>1.$ The estimate (\ref{qcestimate8}) is then a simple consequence of the
monotonocity and
\begin{align*}
\lim_{\rho\rightarrow1}H_{d}\left(  \rho\right)   &  =\lim_{\rho\rightarrow
1}\frac{B_{D}\left(  \rho\right)  +1-\frac{D+1}{D}\left(  DB_{D}\left(
\rho\right)  \right)  ^{\frac{1}{D+1}}}{\left(  1-\rho\right)  ^{2}}\\
&  =\lim_{\rho\rightarrow1}B_{D}^{\prime}\left(  \rho\right)  \lim
_{\rho\rightarrow1}\frac{1-\left(  DB_{D}\left(  \rho\right)  \right)
^{\frac{1}{D+1}-1}}{\left(  -2\right)  \left(  1-\rho\right)  }=\frac{1}%
{2}\left(  D+1\right)  .
\end{align*}
where we used the rule of L'Hospital for $\rho\rightarrow1$ twice, and that
$\frac{1}{2d}\frac{1}{2}\left(  D+1\right)  =\frac{1}{8}.$

3. It remains to show that $N_{d}\left(  \rho\right)  $ is positive on
$\left(  0,1\right)  .$ A short computation shows that for $\rho\in\left(
0,1\right)  $
\begin{equation}
G_{d}^{\prime}=B_{d}^{\prime}-\left(  DB_{d}\right)  ^{\frac{1}{D+1}-1}%
B_{d}^{\prime}=B_{d}^{\prime}\left(  1-\left(  DB_{d}\right)  ^{\frac{1}%
{D+1}-1}\right)  . \label{eqGder}%
\end{equation}
Lemma \ref{LemMain} shows that $B_{d}^{\prime}\left(  \rho\right)  $ converges
for $\rho\rightarrow1$ and $\lim_{\rho\rightarrow1}DB_{d}\left(  \rho\right)
=1,$ hence $G_{d}^{\prime}\left(  \rho\right)  \rightarrow0$ for
$\rho\rightarrow1.$ For $\rho\in\left(  0,1\right)  $ we have
\begin{align*}
N_{d}^{\prime}\left(  \rho\right)   &  =G_{d}^{\prime\prime}\left(
\rho\right)  \left(  1-\rho\right)  +G_{d}^{\prime}\left(  \rho\right) \\
N_{d}^{\prime\prime}\left(  \rho\right)   &  =G_{d}^{\prime\prime\prime
}\left(  \rho\right)  \left(  1-\rho\right)  .
\end{align*}
If we can show that $G_{d}^{\prime\prime\prime}\left(  \rho\right)  >0$ for
all $\rho\in\left(  0,1\right)  $ we see that $N_{d}^{\prime\prime}\left(
\rho\right)  >0$ for all $\rho\in\left(  0,1\right)  ,$ hence $N_{d}^{\prime
}\left(  \rho\right)  $ is increasing on $\left(  0,1\right)  .$ Further
\[
\lim_{\rho\rightarrow1}N_{d}^{\prime}\left(  \rho\right)  =\lim_{\rho
\rightarrow1}G_{d}^{\prime\prime}\left(  \rho\right)  \left(  1-\rho\right)
+\lim_{\rho\rightarrow1}G_{d}^{\prime}\left(  \rho\right)  =0
\]
where we use the fact that $\lim_{\rho\rightarrow1}G_{d}^{\prime\prime}\left(
\rho\right)  $ exists (see below for justification). It follows that
$N_{d}^{\prime}\left(  \rho\right)  \leq0$ for all $\rho\in\left(  0,1\right)
,$ so $N_{d}$ is decreasing. It follows that $N_{d}\left(  \rho\right)  \geq
N_{d}\left(  1\right)  =0.$

4. It remains to show that for $D>1$
\[
G_{d}^{\prime\prime\prime}\text{ is positive on }\left(  0,1\right)  \text{
and }\lim_{\rho\rightarrow1}G_{d}^{\prime\prime}\left(  \rho\right)  \text{
exists. }%
\]
We differentiate (\ref{eqGder}) and using that $\frac{1}{D+1}-1=-\frac{D}%
{D+1}$ we obtain
\[
G_{d}^{\prime\prime}=B_{d}^{\prime\prime}-B_{d}^{\prime\prime}\left(
DB_{d}\right)  ^{\frac{1}{D+1}-1}+\frac{D^{2}}{D+1}B_{d}^{\prime2}\left(
DB_{d}\right)  ^{\frac{1}{D+1}-2}.
\]
Lemma \ref{LemMain} shows that $\lim_{\rho\rightarrow1}B_{d}^{\left(
j\right)  }\left(  \rho\right)  $ exists for $j=0,1,2,$ hence $\lim
_{\rho\rightarrow1}G_{d}^{\prime\prime}\left(  \rho\right)  $ exists. Further
\begin{align*}
G_{d}^{\prime\prime\prime}  &  =B_{d}^{\prime\prime\prime}-B_{d}^{\prime
\prime\prime}\left(  DB_{d}\right)  ^{\frac{1}{D+1}-1}+3B_{d}^{\prime\prime
}B_{d}^{\prime}\frac{D^{2}}{D+1}\left(  DB_{d}\right)  ^{\frac{1}{D+1}-2}\\
&  +\frac{D^{3}}{D+1}B_{d}^{\prime3}\left(  \frac{1}{D+1}-2\right)  \left(
DB_{d}\right)  ^{\frac{1}{D+1}-3}.
\end{align*}
We multiply $G_{d}^{\prime\prime\prime}$ with $\left(  DB_{d}\right)
^{3-\frac{1}{D+1}}$ and we obtain
\[
\left(  DB_{d}\right)  ^{3-\frac{1}{D+1}}G_{d}^{\prime\prime\prime}=\left(
DB_{d}\right)  ^{3-\frac{1}{D+1}}B_{d}^{\prime\prime\prime}+\widetilde{A}%
_{d}\left(  \rho\right)
\]
with
\[
\widetilde{A}_{d}\left(  \rho\right)  =-D^{2}B_{d}^{\prime\prime\prime}%
B_{d}^{2}+3B_{d}^{\prime\prime}B_{d}^{\prime}B_{d}\frac{D^{3}}{D+1}%
-\frac{D^{3}\left(  2D+1\right)  }{\left(  D+1\right)  ^{2}}B_{d}^{\prime3}.
\]
For $D>1$ Lemma \ref{LemMain} shows that $B_{d}^{\prime\prime\prime}\geq0$ on
$\left(  0,1\right)  .$ Hence it suffices to show that $\widetilde{A}%
_{d}\left(  \rho\right)  $ is positive -- which is a much easier function that
$G_{d}^{\prime\prime\prime}.$

5. By dividing $\widetilde{A}_{d}\left(  \rho\right)  $ by $D^{2}$ and
multiplying by $\left(  D+1\right)  $ it remains to show that
\[
A_{d}:=-\left(  D+1\right)  B_{d}^{\prime\prime\prime}B_{d}^{2}+3DB_{d}%
^{\prime\prime}B_{d}^{\prime}B_{d}-DB_{d}^{\prime3}\frac{2D+1}{D+1}%
\]
is positive. According to Lemma \ref{LemMain} we have
\[
\frac{B_{d}^{\prime}\left(  \rho\right)  }{B_{d}\left(  \rho\right)  }%
=\frac{2\left(  D-\rho^{2}+\rho^{2D+2}-\rho^{2}D\right)  }{\rho\left(
1-\rho^{2D}\right)  \left(  1-\rho^{2}\right)  }.
\]
Let $p$ be the numerator and $q$ the denominator. Then $B_{d}^{\prime}\left(
\rho\right)  =B_{d}\left(  \rho\right)  \frac{p}{q}$ and
\begin{align*}
B_{d}^{\prime\prime}  &  =B_{d}^{\prime}\frac{p}{q}+B_{d}\frac{d}{d\rho}%
\frac{p}{q}=B_{d}\left(  \frac{p^{2}}{q^{2}}+\frac{d}{d\rho}\frac{p}{q}\right)
\\
B_{d}^{\prime\prime\prime}  &  =B_{d}^{\prime}\left(  \frac{p^{2}}{q^{2}%
}+\frac{d}{d\rho}\frac{p}{q}\right)  +B_{d}\frac{d}{d\rho}\left(  \left(
\frac{p}{q}\right)  ^{2}+\frac{d}{d\rho}\frac{p}{q}\right) \\
&  =B_{d}\left(  \frac{p^{3}}{q^{3}}+3\frac{p}{q}\frac{d}{d\rho}\frac{p}%
{q}+\frac{d^{2}}{d\rho^{2}}\frac{p}{q}\right)  .
\end{align*}
Replace these expressions in the definition of $A_{d}\left(  \rho\right)  $
and factor out $B_{d}^{3}$. Then we have to show that
\begin{align*}
\widehat{A}_{d}\left(  \rho\right)   &  :=-\left(  D+1\right)  \left(
\frac{p^{3}}{q^{3}}+3\frac{p}{q}\frac{d}{d\rho}\frac{p}{q}+\frac{d^{2}}%
{d\rho^{2}}\frac{p}{q}\right) \\
&  +3D\left(  \frac{p^{2}}{q^{2}}+\frac{d}{d\rho}\frac{p}{q}\right)  \frac
{p}{q}-D\frac{p^{3}}{q^{3}}\frac{2D+1}{D+1}%
\end{align*}
is positive on $\left(  0,1\right)  .$ Simplification gives
\[
\widehat{A}_{d}\left(  \rho\right)  =-\left(  D+1\right)  \left(  \frac{d^{2}%
}{d\rho^{2}}\frac{p}{q}\right)  -3\frac{p}{q}\frac{d}{d\rho}\frac{p}{q}%
-\frac{1}{D+1}\frac{p^{3}}{q^{3}}\geq0
\]
A calculation (e.g with Maple) gives that $\ $%
\[
\widehat{A}_{d}\left(  \rho\right)  =\frac{Da_{d}\left(  \rho^{2}\right)
}{\rho^{3}\left(  D+1\right)  \left(  \rho^{2D}-1\right)  ^{3}\left(  \rho
^{2}-1\right)  ^{3}}%
\]
where $a_{d}\left(  x\right)  $ is equal to
\begin{align*}
&  a_{d}\left(  x\right) \\
=  &  \allowbreak4\left(  D-1\right)  +\allowbreak12x\left(  D+1\right) \\
&  -4x^{D}\left(  2D+1\right)  \left(  D+1\right)  \left(  2D+D^{2}-2\right)
+12x^{D+1}\left(  -4D+5D^{2}+7D^{3}+2D^{4}-2\right)  \allowbreak\\
&  -12x^{D+2}\left(  D+1\right)  \left(  2D+5D^{2}+2D^{3}+1\right)
+\allowbreak4x^{D+3}\left(  2D+1\right)  \left(  D+1\right)  ^{3}\allowbreak\\
&  -4x^{2D}\left(  2D+1\right)  \left(  D+1\right)  ^{3}+12x^{2D+1}\left(
D+1\right)  \left(  2D+5D^{2}+2D^{3}+1\right) \\
&  -12x^{2D+2}\left(  -4D+5D^{2}+7D^{3}+2D^{4}-2\right)  +\allowbreak
4x^{2D+3}\left(  2D+1\right)  \left(  D+1\right)  \left(  2D+D^{2}-2\right) \\
&  -12x^{3D+2}\left(  D+1\right)  -4x^{3D+3}\left(  D-1\right)  .
\end{align*}
Note that $a_{d}\left(  x\right)  $ is a linear combination of $12$ power
functions. It can be shown that the function $a_{d}\left(  x\right)  $ has a
zero of order $7$ at $x=1.$ Further
\[
\allowbreak\lim_{x\rightarrow1}\frac{d^{7}}{dx^{7}}a_{d}\left(  x\right)
=-168D^{3}\left(  D-1\right)  \left(  2D+1\right)  \left(  D+2\right)  \left(
D+1\right)  ^{3}<0
\]
Note that $a_{d}\left(  0\right)  >0$. If $a_{d}\left(  x_{0}\right)  <0$ for
some $x_{0}\in\left(  0,1\right)  $ we see that $a_{d}$ has an a zero in
$\left(  0,x_{0}\right)  $. Since $a_{d}\left(  x\right)  >0$ for $x<1$ close
enough to $1$ we see that $a_{d}$ has also a zero in $\left(  x_{0},1\right)
.$ Due to the symmetry of the function
\[
-x^{3D+3}a_{d}\left(  \frac{1}{x}\right)  =a_{d}\left(  x\right)
\]
we infer that the function $a_{d}$ has at least $11$ positive zeros. This is
impossible since the coefficients of the function $a_{d}$ have only $9$
changes of sign.
\end{proof}

\bigskip ACKNOWLEDGEMENTS \medskip The work of H. Render and Ts. Tsachev was
funded under project KP-06-N32-8, while the work of O. Kounchev was funded
under project KP-06-N42-2 with Bulgarian NSF.

\end{document}